\documentclass[11pt]{amsart} 

\usepackage{amsmath,amssymb,amsthm}
\usepackage{mathtools}
\usepackage{graphicx} 
\usepackage{hyperref}
\usepackage{array}
\usepackage[usenames,dvipsnames]{xcolor}
\usepackage{tikz}
	\usetikzlibrary{automata, positioning}
\usepackage[noabbrev,capitalise]{cleveref}

\usepackage[left=1.5in,right=1.5in,top=1.5in,bottom=1.5in]{geometry}
\usepackage{stix}
\usepackage{xspace}

\theoremstyle{definition} 
	\newtheorem{definition}{Definition}[section]
	\newtheorem{example}[definition]{Example}

\theoremstyle{plain}
	\newtheorem{theorem}[definition]{Theorem}
	
	\newtheorem{proposition}[definition]{Proposition}
	\newtheorem{lemma}[definition]{Lemma}

\newcommand{\oneslant}[1]{\big[T_{n,#1}\big]\diagdown}

\newcommand{\leftaddoneboth}[1]{\mathchoice
        {\mathord{\raisebox{2pt}{$\diagdown$}\hspace{-2pt}\raisebox{-1pt}{$\llcorner$}}}
        {\mathord{\raisebox{2pt}{$\diagdown$}\hspace{-2pt}\raisebox{-1pt}{$\llcorner$}}}
        {\mathord{\raisebox{2pt}{$\scriptscriptstyle\diagdown$}\hspace{-1pt}\raisebox{-3pt}{$\scriptscriptstyle\llcorner$}}}
        {\mathord{\raisebox{2pt}{$\scriptscriptstyle\diagdown$}\hspace{-1pt}\raisebox{-3pt}{$\scriptscriptstyle\llcorner$}}}
    \hspace{-1pt}\big[T_{n,#1}\big] \hspace{-1pt}\mathchoice
        {\mathord{\raisebox{2pt}{$\diagdown$}\hspace{-5pt}\raisebox{-1pt}{$\lrcorner$}}}
        {\mathord{\raisebox{2pt}{$\diagdown$}\hspace{-5pt}\raisebox{-1pt}{$\lrcorner$}}}
        {\mathord{\raisebox{2pt}{$\scriptscriptstyle\diagdown$}\hspace{-5pt}\raisebox{-1pt}{$\scriptscriptstyle\lrcorner$}}}
        {\mathord{\raisebox{2pt}{$\scriptscriptstyle\diagdown$}\hspace{-5pt}\raisebox{-1pt}{$\scriptscriptstyle\lrcorner$}}}
    \hspace{1pt}
}

\newcommand{\bothaddone}[1]{\mathchoice
        {\mathord{\raisebox{2pt}{$\diagdown$}\hspace{-1pt}\raisebox{-1pt}{$\llcorner$}}}
        {\mathord{\raisebox{2pt}{$\diagdown$}\hspace{-1pt}\raisebox{-1pt}{$\llcorner$}}}
        {\mathord{\raisebox{2pt}{$\scriptscriptstyle\diagdown$}\hspace{-1pt}\raisebox{-1pt}{$\scriptscriptstyle\llcorner$}}}
        {\mathord{\raisebox{2pt}{$\scriptscriptstyle\diagdown$}\hspace{-1pt}\raisebox{-1pt}{$\scriptscriptstyle\llcorner$}}}
    \hspace{-1pt}\big[T_{n,#1}\big] \hspace{-1pt}\mathchoice
        {\mathord{\raisebox{2pt}{$\urcorner$}\hspace{-1pt}\raisebox{-1pt}{$\diagdown$}}}
        {\mathord{\raisebox{2pt}{$\urcorner$}\hspace{-1pt}\raisebox{-1pt}{$\diagdown$}}}
        {\mathord{\raisebox{2pt}{$\scriptscriptstyle\urcorner$}\hspace{-1pt}\raisebox{-1pt}{$\scriptscriptstyle\diagdown$}}}
        {\mathord{\raisebox{2pt}{$\scriptscriptstyle\urcorner$}\hspace{-1pt}\raisebox{-1pt}{$\scriptscriptstyle\diagdown$}}}
}

\newcommand{\rightaddonly}[1]{
    \big[T_{n,#1}\big]
    \hspace{-1pt}\mathchoice
        {\mathord{\raisebox{2pt}{$\urcorner$}\hspace{-1pt}\raisebox{-1pt}{$\diagdown$}}}
        {\mathord{\raisebox{2pt}{$\urcorner$}\hspace{-1pt}\raisebox{-1pt}{$\diagdown$}}}
        {\mathord{\raisebox{2pt}{$\scriptscriptstyle\urcorner$}\hspace{-1pt}\raisebox{-1pt}{$\scriptscriptstyle\diagdown$}}}
        {\mathord{\raisebox{2pt}{$\scriptscriptstyle\urcorner$}\hspace{-1pt}\raisebox{-1pt}{$\scriptscriptstyle\diagdown$}}}
}

\newcommand{\rightminusonly}[1]{
    \big[T_{n,#1}\big]
    \hspace{-0.1em}\mathchoice
        {\mathord{\raisebox{2pt}{$\diagdown$}\hspace{-5pt}\raisebox{-1pt}{$\lrcorner$}}}
        {\mathord{\raisebox{2pt}{$\diagdown$}\hspace{-5pt}\raisebox{-1pt}{$\lrcorner$}}}
        {\mathord{\raisebox{2pt}{$\scriptscriptstyle\diagdown$}\hspace{-5pt}\raisebox{-1pt}{$\scriptscriptstyle\lrcorner$}}}
        {\mathord{\raisebox{2pt}{$\scriptscriptstyle\diagdown$}\hspace{-5pt}\raisebox{-1pt}{$\scriptscriptstyle\lrcorner$}}}
    \hspace{1pt}
}

\newcommand{\bothminusone}[1]{\mathchoice
        {\mathord{\raisebox{2pt}{$\ulcorner$}\hspace{-4pt}\raisebox{-1pt}{$\diagdown$}}}
        {\mathord{\raisebox{2pt}{$\ulcorner$}\hspace{-4pt}\raisebox{-1pt}{$\diagdown$}}}
        {\mathord{\raisebox{2pt}{$\scriptscriptstyle\ulcorner$}\hspace{-4pt}\raisebox{-1pt}{$\scriptscriptstyle\diagdown$}}}
        {\mathord{\raisebox{2pt}{$\scriptscriptstyle\ulcorner$}\hspace{-4pt}\raisebox{-1pt}{$\scriptscriptstyle\diagdown$}}}
    \hspace{-1pt}\big[T_{n,#1}\big] \hspace{-0.1em}\mathchoice
        {\mathord{\raisebox{2pt}{$\diagdown$}\hspace{-5pt}\raisebox{-1pt}{$\lrcorner$}}}
        {\mathord{\raisebox{2pt}{$\diagdown$}\hspace{-5pt}\raisebox{-1pt}{$\lrcorner$}}}
        {\mathord{\raisebox{2pt}{$\scriptscriptstyle\diagdown$}\hspace{-5pt}\raisebox{-1pt}{$\scriptscriptstyle\lrcorner$}}}
        {\mathord{\raisebox{2pt}{$\scriptscriptstyle\diagdown$}\hspace{-5pt}\raisebox{-1pt}{$\scriptscriptstyle\lrcorner$}}}
    \hspace{1pt}
}

\newcommand{\basegraph}[1]{T_{n,#1}}

\newcommand{\ruleset}[1]{\textup{\textsc{#1}}\xspace}
\newcommand{\snort}{\ruleset{Snort}}

\newcommand{\CGTN}{{\mathscr{N}}}
\newcommand{\CGTP}{{\mathscr{P}}}
\newcommand{\CGTL}{{\mathscr{L}}}
\newcommand{\CGTR}{{\mathscr{R}}}

\begin{document}

\title{\snort Played on Triangular Grids}

\author{Melanie Gauthier}
\address{Mount Saint Vincent University, Canada}

\author{Svenja Huntemann}
\address{Mount Saint Vincent University, Canada}
\email{svenja.huntemann@msvu.ca}


\thanks{The first author's research was supported by the Natural Sciences and Engineering Research Council of Canada through the Undergraduate Student Research Award 590979-2024 and the Atlantic Association for Research in the Mathematical Sciences through the collaborative research group ``Games and Graph Searching in Atlantic Canada''. The second author's research is supported in part by the Natural Sciences and Engineering Research Council of Canada grant 2022-04273.}

\begin{abstract}
    \snort is a two-player game played on a simple graph in which the players take turns colouring vertices in their own colour, with the restriction that two adjacent vertices cannot have opposite colours. We will show that on triangular grids with one or two rows of triangles, and many of their variants, the first player will win when playing optimally.
\end{abstract}

\maketitle

\section{Introduction}

\snort is a placement game played on any finite graph in which two players, called Left and Right, alternately colour vertices. Left colours vertices blue, while Right colours red, such that opposite colours are not adjacent. We will assume the normal play convention, so that the game is over when the next player has no available moves, and the person who played last wins. \snort was introduced by Simon Norton (see \cite{Conway2001,WW}) and is also known as Cats and Dogs.

Schaefer proved in 1978 that finding the winning player of \snort in general is PSPACE-complete \cite{SchaeferPaper}, thus further analysis has restricted which types of graphs to consider. Winning Ways \cite{WW} contains a collection of various \snort positions and their game values, which intuitively indicate which player will win and how much of an advantage they have. Kakihara outlines some basic strategies for \snort in their thesis \cite{Kakihara2010}. Other work has enumerated the number of legal positions when playing \snort on simple graphs \cite{Brown_Cox_ETAL_2019,Lexi_2022}. Uiterwijk solved \snort on Cartesian grids \cite{Uiterwijk_2022_P1} by finding the outcome class, i.e.\ which player wins using optimal strategies.

Motivated by Uiterwijk's work on Cartesian grid, we will be looking at \snort on triangular grids and determining what the outcome classes and some optimal strategies are.

Some of the strategies considered by Kakihara include playing greedy by attempting to reserve as many vertices for yourself as possible, blocking your opponent from reserving more vertices by blocking them on both sides, and when there is symmetry the second player can use a copy cat strategy. We use these ideas and basic strategies in our work, in addition to having used software (\cite{CGSuite,TomRust}) to aid in computing outcome classes and optimal moves in some cases.

Since opposite colours cannot be adjacent, when a player colours a vertex, they also ``reserve'' all neighbouring vertices. We will indicate this by giving those vertices a tinting in light blue or light red.

To simplify the analysis of a game of \snort, we will reduce the graph after each move. First, we will be deleting the claimed vertex after tinting the adjacent vertices. Additionally, a vertex with both a blue and a red neighbour (tinted blue and red) cannot be claimed by either player, so can also be deleted.  

\begin{example}
    \label{ex:SNORT}
    Consider \snort played on $P_6$ with Left playing first, as shown in the diagram below, with some examples moves from both players.
    
    Left chooses to go in one of the two central vertices. This causes that vertex to be deleted from the graph and the adjacent ones tinted blue.
    
    Right responds by choosing to go in vertex 5. We will again delete this vertex, and the ones adjacent are tinted red. However vertex 4 was tinted blue already, which means that neither player can go there and it is also deleted. 
    
    This leaves Left to go in vertex 2, which in turn tints vertex 1, and now Right and Left take turns colouring their tinted vertices. Since Left played last, Left therefore wins in this instance of the game.
    
    
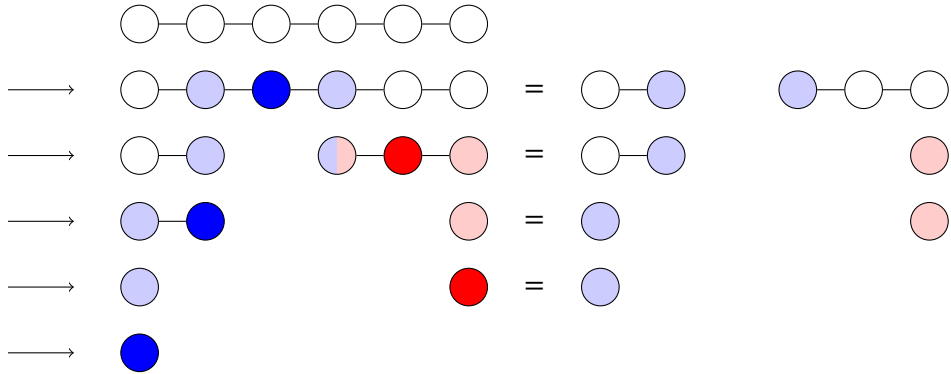
\begin{center}
    \begin{tikzpicture}[scale=1.25, vertex/.style={circle, draw, minimum size=5mm, font=\scriptsize},scale=0.70]
    \begin{scope}[shift={(2,0)}]
        \node[vertex] (1) at (0,0) {};
        \node[vertex] (2) at (1,0) {} edge (1);
        \node[vertex] (3) at (2,0) {} edge (2);
        \node[vertex] (4) at (3,0) {} edge (3);
        \node[vertex] (5) at (4,0) {} edge (4);
        \node[vertex] (6) at (5,0) {} edge (5);
    \end{scope}
    \draw[->] (0,-1)--(1,-1);
    \begin{scope}[shift={(2,-1)}]
        \node[vertex] (1) at (0,0) {};
        \node[vertex,fill=blue!20] (2) at (1,0) {} edge (1);
        \node[vertex,fill=blue] (3) at (2,0) {} edge (2);
        \node[vertex,fill=blue!20] (4) at (3,0) {} edge (3);
        \node[vertex] (5) at (4,0) {} edge (4);
        \node[vertex] (6) at (5,0) {} edge (5);
    \end{scope}
        \node at (8,-1) {=};
    \begin{scope}[shift={(9,-1)}]
        \node[vertex] (1) at (0,0) {};
        \node[vertex,fill=blue!20] (2) at (1,0) {} edge (1);
        \node[vertex,fill=blue!20] (4) at (3,0) {};
        \node[vertex] (5) at (4,0) {} edge (4);
        \node[vertex] (6) at (5,0) {} edge (5);
    \end{scope}
    \draw[->] (0,-2)--(1,-2);
    \begin{scope}[shift={(2,-2)}]
        \node[vertex] (1) at (0,0) {};
        \node[vertex,fill=blue!20] (2) at (1,0) {} edge (1);
        \node[vertex, path picture={
            \fill[blue!20] (path picture bounding box.south west) rectangle (path picture bounding box.north);
            \fill[red!20] (path picture bounding box.south) rectangle (path picture bounding box.north east);
        }] (4) at (3,0) {};
        \node[vertex,fill=red] (5) at (4,0) {} edge (4);
        \node[vertex,fill=red!20] (6) at (5,0) {} edge (5);
    \end{scope}
    \node at (8,-2) {=};
    \begin{scope}[shift={(9,-2)}]
        \node[vertex] (1) at (0,0) {};
        \node[vertex,fill=blue!20] (2) at (1,0) {} edge (1);
        \node[vertex,fill=red!20] (6) at (5,0) {};
    \end{scope}
    \draw[->] (0,-3)--(1,-3);
    \begin{scope}[shift={(2,-3)}]
        \node[vertex,fill=blue!20] (1) at (0,0) {};
        \node[vertex,fill=blue] (2) at (1,0) {} edge (1);
        \node[vertex,fill=red!20] (6) at (5,0) {};
    \end{scope}
        \node at (8,-3) {=};
    \begin{scope}[shift={(9,-3)}]
        \node[vertex,fill=blue!20] (1) at (0,0) {};
        \node[vertex,fill=red!20] (6) at (5,0) {};
    \end{scope}
    \draw[->] (0,-4)--(1,-4);
    \begin{scope}[shift={(2,-4)}]
        \node[vertex,fill=blue!20] (1) at (0,0) {};
        \node[vertex,fill=red] (6) at (5,0) {};
    \end{scope}
        \node at (8,-4) {=};
    \begin{scope}[shift={(9,-4)}]
        \node[vertex,fill=blue!20] (1) at (0,0) {};
    \end{scope}
    \draw[->] (0,-5)--(1,-5);
    \begin{scope}[shift={(2,-5)}]
        \node[vertex,fill=blue] (1) at (0,0) {};
    \end{scope}
    \end{tikzpicture}
    \end{center}
\end{example}

When both players play optimally there are four types of results we can get, called \emph{outcome classes}. They are as follows:
\begin{itemize}
        \item $\CGTN$: games in which the first/next player wins no matter who goes first;
        \item $\CGTP$: games in which the second/previous player wins no matter who goes first;
        \item $\CGTL$: games in which Left wins going first or second; and
        \item $\CGTR$: games in which Right wins going first or second.
\end{itemize}

In \cref{ex:SNORT}, both players are playing optimally. Right going first can use the same strategy, thus \snort played on $P_6$ is a first player win.  

\medskip

There are several different ways of displaying and describing triangular grids. We will use the triangulated Cartesian product of two paths as our base and build different types of triangular grids from there. 

The triangulated Cartesian product $G\boxbslash H$ is defined to be the graph with vertex set $V(G)\times V(H)$ and with the edge set consisting of the edges of the Cartesian product of $G$ and $H$, but with one diagonal edge for each Cartesian square \cite{Afzal2012}. It is of note that the direction of the diagonal edge in the definition in \cite{Afzal2012} is unspecified -- for our purposes all graphs will have the diagonal edge going in the same direction.

We will be considering the triangulated Cartesian product of two paths, which results in a Cartesian rectangle with diagonals added. We will also look at variants where some vertices are added to one or both ends such that new triangles are formed without extending the rectangle itself. We will label the vertices in the rectangle with their coordinates, while the vertices added on the left and right are labeled as $L_i$ or $R_i$ (see \cref{fig:label}), where $L_i=(0,i)$ and $R_i=(n+1,i)$.

\begin{figure}[!ht]
    \centering
    \begin{tikzpicture}[scale=.7, vertex/.style={circle, draw, minimum size=3mm, font=\scriptsize}]
        
        \node[vertex,dashed] (1) at (0,4) {\small$L_1$};
        \node[vertex] (2) at (2,4) {$(1,1)$} edge[dashed] (1);
        \node[vertex] (3) at (4,4) {$(2,1)$} edge (2);
        \node[vertex] (4) at (6,4) {$(3,1)$} edge (3);
        \node[vertex] (5) at (8,4) {$(4,1)$} edge (4);
        \node[vertex] (6) at (10,4) {$(5,1)$} edge (5);
        \node[vertex,dashed] (7) at (12,4) {\small$R_1$} edge[dashed] (6);
        
        \node[vertex,dashed] (8) at (0,2) {\small$L_2$} edge[dashed] (1);
        \node[vertex] (9) at (2,2) {$(1,2)$} edge[dashed] (8) edge[dashed] (1) edge (2);
        \node[vertex] (10) at (4,2) {$(2,2)$} edge (9) edge (3) edge (2);
        \node[vertex] (11) at (6,2) {$(3,2)$} edge (10) edge (4) edge (3);
        \node[vertex] (12) at (8,2) {$(4,2)$} edge (11) edge (5) edge (4);
        \node[vertex] (13) at (10,2) {$(5,2)$} edge (12) edge (6) edge (5);
        \node[vertex,dashed] (14) at (12,2) {\small$R_2$} edge[dashed] (13) edge[dashed] (7) edge[dashed] (6);
        
        \node[vertex,dashed] (15) at (0,0) {\small$L_3$} edge[dashed] (8);   
        \node[vertex] (16) at (2,0) {$(1,3)$} edge[dashed] (15) edge[dashed] (8) edge (9);
        \node[vertex] (17) at (4,0) {$(2,3)$} edge (16) edge (10) edge (9);
        \node[vertex] (18) at (6,0) {$(3,3)$} edge (17) edge (11) edge (10);
        \node[vertex] (19) at (8,0) {$(4,3)$} edge (18) edge (12) edge (11);
        \node[vertex] (20) at (10,0) {$(5,3)$} edge (19) edge (13) edge (12);
        \node[vertex,dashed] (21) at (12,0) {\small$R_3$} edge[dashed] (20) edge[dashed] (14) edge[dashed] (13);
        
        \node[vertex] (22) at (2,6) {};
        \node[vertex] (23) at (4,6) {} edge (22);
        \node[vertex,label=above:$P_5$] (24) at (6,6) {} edge (23);
        \node[vertex] (25) at (8,6) {} edge (24);
        \node[vertex] (26) at (10,6) {} edge (25);
        \node[vertex] (27) at (-2,0) {};
        \node[vertex,label=left:$P_3$] (28) at (-2,2) {} edge (27);
        \node[vertex] (29) at (-2,4) {} edge (28);
        
        \node[vertex,dashed] (30) at (14,4) {\small$R_1'$} edge[dashed] (7);   
        \node[vertex,dashed] (31) at (14,2) {\small$R_2'$} edge[dashed] (30) edge[dashed] (14) edge[dashed] (7);
        \node[vertex,dashed] (32) at (14,0) {\small$R_3'$} edge[dashed] (31) edge[dashed] (21) edge[dashed] (14);
    \end{tikzpicture}
    \caption{Labeling of vertices in $P_5\boxbslash P_3$}
    \label{fig:label}
\end{figure}

Formally, we define the triangular grid $T_{n,m}$ as follows:
\begin{definition}
    Let \(T_{n,m}=P_n\boxbslash P_m\) be the graph with vertex set \[V=\{(i,j)\mid i=1,\ldots,n \text{ and } j=1,\ldots,m\}\] and edges
    \begin{align*}
        E=&\hspace{1em}\left\{\left((i_1,j),(i_2,j)\right)\mid i_2=i_1+1,\; i_1=1,\ldots,n-1,\; j=1,\ldots,m \right\}\\
        &\cup\left\{\left((i,j_1),(i,j_2)\right)\mid j_2=j_1+1,\; i=1,\ldots,n,\; j_1=1,\ldots,m-1 \right\}\\
        &\cup\left\{\left((i_1,j_1),(i_2,j_2)\right)\mid i_2=i_1+1,\; j_2=j_1+1,\; i_1=1,\ldots,n-1,\; j=1,\ldots,m \right\}.
    \end{align*}
\end{definition}

We define several variants where vertices and corresponding edges are added to $T_{n,2}$ and $T_{n,3}$ next:
\begin{definition}
    The variant $\oneslant{2}$ has vertex set $V=V(T_{n,2})\cup\{R_2\}$ and edge set $E=E(T_{n,2})\cup\left\{\left((n,1),R_2\right),\left((n,2),R_2\right)\right\}$.

    \medskip
    For the below variants of $\basegraph{3}$ the vertex set is $V(T_{n,3})\cup V'$ and the edge set is $E(T_{n,3})\cup E'$.
    \begin{itemize}
        \item $\oneslant{3}$: $V'=\{R_2, R_3, R_3'\}$ and \[E'=\big\{\left((n,1),R_2\right),\left((n,2),R_2\right),\left((n,2),R_3\right),\left((n,3),R_3\right),(R_2,R_3),(R_2,R_3'),(R_3,R_3')\big\}.\]
        \item $\leftaddoneboth{3}$: $V'=\{L_1,R_2,R_3\}$ and 
        \begin{align*}
            E'=\big\{\left(L_1, (1,1)\right),\left(L_1, (1,2)\right),\left((n,1),R_2\right),\left((n,2),R_2\right),\\\left((n,2),R_3\right),\left((n,3),R_3\right),(R_2,R_3)\big\}.
        \end{align*}
        \item $\bothaddone{3}$: $V'=\{L_1,R_3\}$ and \[E'=\big\{\left(L_1, (1,1)\right),\left(L_1, (1,2)\right),\left((n,2),R_3\right),\left((n,3),R_3\right)\big\}.\]
        \item $\rightaddonly{3}$: $V'=\{R_3\}$ and \[E'=\big\{\left((n,2),R_3\right),\left((n,3),R_3\right)\big\}.\]
        \item $\rightminusonly{3}$: 
        $V'=\{R_2,R_3\}$ and \[E'=\big\{\left((n,1),R_2\right),\left((n,2),R_2\right),\left((n,2),R_3\right),\left((n,3),R_3\right),\left(R_2,R_3\right)\big\}.\]
        \item $\bothminusone{3}$: 
        $V'=\{L_1,L_2,R_2,R_3\}$ and 
        \begin{align*}
        E'=\big\{\left(L_1, (1,1)\right),\left(L_1, (1,2)\right),\left(L_2, (1,2)\right),\left(L_2, (1,3)\right),\left(L_1, L_2\right),\\\left((n,1),R_2\right),\left((n,2),R_2\right),\left((n,2),R_3\right),\left((n,3),R_3\right),\left(R_2,R_3\right)  \big\}.
        \end{align*}
    \end{itemize}
\end{definition}

In this paper we only consider variations of the base graph that allow for us to be able to add the diagonal edges for a full ``triangle'', i.e.\ those that are connected induced subgraphs of a larger triangular Cartesian grid with the given set of vertices. Some of the variations that would result in incomplete triangles are:
\begin{itemize}
    \item for $\basegraph{2}$: adding $L_2$ or $R_1$, or adding $R_2$ and $R_2'$;
    \item for $\basegraph{3}$: adding $L_2$ or $R_2$, adding $L_3$ or $R_1$, adding $L_2$ and $L_3$ or $R_1$ and $R_2$, or adding $R_2$, $R_3$, $R_2'$ and $R_3'$.
\end{itemize}

Note that with this restriction the variant $\oneslant{2}$ is the only one we need to consider for $\basegraph{2}$ since adding $L_1$ is equivalent by symmetry to adding $R_2$. Similarly, we will only consider the six variants listed for $\basegraph{3}$.

Note that when describing winning strategies for playing \snort on these graphs, it is often easier to see what the strategy is by drawing the graph with equilateral triangles -- we will show both drawings in those cases.

\medskip

For the majority of the proofs we will utilize a copycat strategy, where the second player ``copies'' the first player's move. In our cases, this will be possible since the graph will consists of two isomorphic connected components, except possibly some shading in the colour of the second player. The following proposition shows that such a graph is a second player win.

\begin{proposition}\label{thm:symmetry}
    If the graph $H$ is a disjoint union of $G$ and $G'$, where $G$ and $G'$ are isomorphic except possibly for some tinting in the colour of the second player, then \snort on $H$ is a second player win using the copycat strategy.
\end{proposition}
\begin{proof}
    Wherever the first player moves, the second player will mimic this move in the other component. Since the only restrictions originally due to tinting are for the first player and the first player's moves are always legal, the second player's moves will also be allowed. Using this strategy, the second player will then make the last move and win.
\end{proof}

In order for us to use the copy cat strategy, we will split the graph using a move by the first player (who then becomes the second player) and the one-hand-tied principle. The one-hand-tied principle states that if a player has a winning strategy in a game $G$ without using some of the options, then they have a winning strategy in the game $G$ (see for example \cite{LiP2019}). 

The first player will think of some of their tinted vertices adjacent to their first move as having been deleted from the graph, resulting in two isomorphic connected components on which they play copycat. Since they can win using this strategy by ignoring some of their options, they can certainly win the overall game.

\section{\snort played on $\basegraph{2}$ and its Variant}

In this section, we will show that \snort played on graphs of the form $\basegraph{2}$ or its variation is a first player win. These graphs can be represented on the triangular grid graph as one row of triangles.  

\begin{proposition}\label{thm:1row}
    When playing \snort on $\basegraph{2}$ or its variation $\oneslant{2}$, the outcome class is a first player win.
\end{proposition}
\begin{proof}
    For graphs where $n<3$ the solution is trivial, and can be checked using software like CGSuite. For graphs where $n\geq3$, the proof is split into four cases, two for each type of graph based on the parity of the length. In each case we will outline where the optimal move is for the first player. We will assume without loss of generality that the first player is Left -- Right's strategy just reverses colours.
    
    After Left's first move, she can think of the graph as split into two symmetric halves by the one hand tied principle, except possibly some additional blue shading in one or both halves. This split will happen naturally if Left ignores some (or all) of the blue tinted vertices in her strategy. 
    
    Once she considers the graph split, using \cref{thm:symmetry} Left can copy Right's moves on the other half. This will ensure that Left's moves are always permitted because Right will always move in a legal position, and thus Left will move last.

    First consider \snort played on $\basegraph{2}$ when $n$ is odd, as shown in  \cref{fig:P2_odd_odd}.
    The optimal first move for Left in this situation is to go in $\left(\lceil n/2\rceil,1\right)$ or in $\left(\lceil n/2\rceil,2\right)$. 
    This allows for Left to think of the graph as if it were split into two equal halves, aside from some shading, by ignoring the vertices $(\lceil n/2\rceil,1)$ and $(\lceil n/2 \rceil,2)$.

    \begin{figure}[!ht]
        \centering
        \begin{tikzpicture}[line cap=round,line join=round,x=1cm,y=1cm,rotate=90, vertex/.style={circle, draw, minimum size=3mm, font=\scriptsize},scale=.8]
        \node [vertex] (1) at (9,8) {};
        \node [vertex] (2) at (9,7) {} edge (1);
        \node [vertex,fill=blue!30] (3) at (9,6) {} edge (2);
        \node[vertex,draw=orange!120,line width=1pt,fill=blue] (4) at (9,5) {} edge (3);
        \node [vertex,fill=blue!30] (5) at (9,4) {} edge (4);
        \node [vertex] (6) at (9,3) {} edge (5);
        \node [vertex] (7) at (9,2) {} edge (6);
        \node [vertex] (8) at (10,2) {} edge (7);
        \node [vertex] (9) at (10,3) {} edge (8) edge (7) edge (6);
        \node [vertex] (10) at (10,4) {} edge (9) edge (6) edge (5);
        \node[vertex,draw=orange!120,line width=1pt,fill=blue!30] (11) at (10,5) {} edge (10) edge (5) edge[orange,line width=2pt] (4);
        \node [vertex,fill=blue!30] (12) at (10,6) {} edge (11) edge (4)  edge (3);
        \node [vertex] (13) at (10,7) {} edge (12) edge (3) edge (2);
        \node [vertex] (14) at (10,8) {} edge (13) edge (2) edge (1);
        \node at (9.5,1.5) {$\cdots$};
        \node at (9.5,8.5) {$\cdots$};
        \end{tikzpicture}
        \caption{First player move on $\basegraph{2}$ when $n$ is odd, and where to split the graph.}
        \label{fig:P2_odd_odd}
    \end{figure}
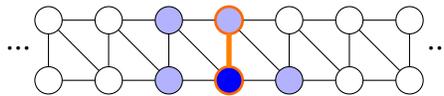

    Next consider when \snort is played on the same graph but $n$ is even, as shown in \cref{fig:P2_even_even}. Here Left should go in $\left(n/2+1,2\right)$, using the same idea as the previous case she should act as if the vertices $(n/2,1)$, $(n/2,2)$, $(n/2+1,1)$, and $(n/2+1,2)$ have been deleted from the graph. 
    
    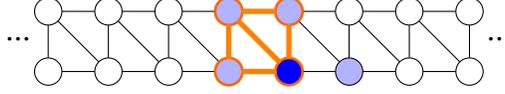
\begin{figure}[!ht]
        \centering
        \begin{tikzpicture}[line cap=round,line join=round,x=1cm,y=1cm,rotate=90, vertex/.style={circle, draw, minimum size=3mm, font=\scriptsize},scale=.8]
        \node [vertex] (1) at  (9,9) {};
        \node [vertex] (2) at  (9,8) {}edge (1);
        \node [vertex] (3) at  (9,7) {}edge (2);
        \node[vertex,draw=orange!120,line width=1pt,fill=blue!30] (4) at  (9,6) {}edge (3);
        \node[vertex,draw=orange!120,line width=1pt,fill=blue] (5) at (9,5) {} edge[orange,line width=2pt] (4);
        \node [vertex,fill=blue!30] (6) at  (9,4) {}edge (5);
        \node [vertex] (7) at  (9,3) {}edge (6);
        \node [vertex] (8) at  (9,2) {}edge (7);
        \node [vertex] (9) at  (10,2) {}edge (8);
        \node [vertex] (10) at  (10,3) {}edge (9) edge (7) edge (8);
        \node [vertex] (11) at  (10,4) {}edge (10)  edge (6) edge (7);
        \node[vertex,draw=orange!120,line width=1pt,fill=blue!30] (12) at  (10,5) {} edge (11) edge[orange,line width=2pt] (5) edge (6);
        \node[vertex,draw=orange!120,line width=1pt,fill=blue!30] (13) at  (10,6) {}edge[orange,line width=2pt] (12) edge[orange,line width=2pt] (4) edge[orange,line width=2pt] (5);
        \node [vertex] (14) at  (10,7) {}edge (13) edge (3) edge (4);
        \node [vertex] (15) at  (10,8) {}edge (14) edge (2) edge (3);
        \node [vertex] (16) at  (10,9) {}edge (15) edge (1) edge (2);
        \node at (9.5,1.5) {$\cdots$};
        \node at (9.5,9.5) {$\cdots$};
        \end{tikzpicture}
        \caption{First player move on $\basegraph{2}$ when $n$ is even, and where to split the graph.}
        \label{fig:P2_even_even}
    \end{figure}
    
    \vspace{5pt}

    Now we can look at the variant $\oneslant{2}$. If $n$ is odd, like in \cref{fig:P2_odd_n}, Left should go in the position $\left(\lceil n/2\rceil,1\right)$. She can once again think of the graph as split into two halves that are mirror symmetric by ignoring the vertex $\left(\lceil n/2\rceil,1\right)$, and the edge $\left(\left(\lceil n/2\rceil,2\right),\left(\lceil n/2\rceil+1,2\right)\right)$. 
    
    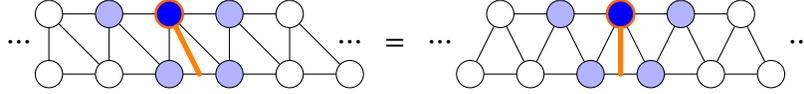
\begin{figure}[!ht]
        \centering
        \begin{tikzpicture}[line cap=round,line join=round,x=1cm,y=1cm,rotate=90, vertex/.style={circle, draw, minimum size=3mm, font=\scriptsize},scale=.8]
        \node [vertex] (1) at (9,8) {};
        \node [vertex] (2) at (9,7) {} edge (1);
        \node [vertex,fill=blue!30] (3) at (9,6) {} edge (2);
        \node [vertex,fill=blue!30] (4) at (9,5) {} edge (3);
        \node [vertex] (5) at (9,4) {} edge (4);
        \node [vertex] (6) at (9,3) {} edge (5);
        \node [vertex] (10) at (10,4) {} edge (6) edge (5);
        \node [vertex,fill=blue!30] (11) at (10,5) {} edge (10) edge (5) edge (4);
        \node[vertex,draw=orange!120,line width=1pt,fill=blue] (12) at (10,6) {} edge (11) edge (4)  edge (3);
        \node [vertex,fill=blue!30] (13) at (10,7) {} edge (12) edge (3) edge (2);
        \node [vertex] (14) at (10,8) {} edge (13) edge (2) edge (1);
        \draw[orange,line width=2pt] (12) -- (9,5.5);
        \node at (9.5,3) {$\cdots$};
        \node at (9.5,8.5) {$\cdots$};
        \node at (9.5,2.25) {$=$};
        \begin{scope}[shift={(0,-7)}]
            \node [vertex] (1) at (9,8) {};
            \node [vertex] (2) at (9,7) {} edge (1);
            \node [vertex,fill=blue!30] (3) at (9,6) {} edge (2);
            \node [vertex,fill=blue!30] (4) at (9,5) {} edge (3);
            \node [vertex] (5) at (9,4) {} edge (4);
            \node [vertex] (6) at (9,3) {} edge (5);
            \node [vertex] (10) at (10,3.5) {} edge (6) edge (5);
            \node [vertex,fill=blue!30] (11) at (10,4.5) {} edge (10) edge (5) edge (4);
            \node[vertex,draw=orange!120,line width=1pt,fill=blue] (12) at (10,5.5) {} edge (11) edge (4)  edge (3);
            \node [vertex,fill=blue!30] (13) at (10,6.5) {} edge (12) edge (3) edge (2);
            \node [vertex] (14) at (10,7.5) {} edge (13) edge (2) edge (1);
            \draw[orange,line width=2pt] (12) -- (9,5.5);
            \node at (9.5,2.5) {$\cdots$};
            \node at (9.5,8.5) {$\cdots$};
        \end{scope}
        \end{tikzpicture}
        \caption{First player move when $\oneslant{2}$ when $n$ is odd, and where to split the graph.}
        \label{fig:P2_odd_n}
    \end{figure}

    Lastly consider \snort played on $\oneslant{2}$ and $n$ is even, as shown in  \cref{fig:P2_even_n}
    The optimal first move for Left in this situation is to go in $\left(n/2+1,2\right)$. 
    This allows for Left to think of the graph as if it were split into two halves which are mirror symmetric, by deleting the vertex $\left(n/2+1,2\right)$, and the edge $\left(\left(n/2,1\right),\left(n/2+1,1\right)\right)$.
    
    \begin{figure}[!ht]
        \centering
        \begin{tikzpicture}[line cap=round,line join=round,x=1cm,y=1cm,rotate=90, vertex/.style={circle, draw, minimum size=3mm, font=\scriptsize},scale=.8]
        \node [vertex] (1) at (9,8) {};
        \node [vertex] (2) at (9,7) {} edge (1);
        \node [vertex,fill=blue!30] (3) at (9,6) {} edge (2);
        \node [vertex,draw=orange!120,line width=1pt,fill=blue] (4) at (9,5) {} edge (3);
        \node [vertex,fill=blue!30] (5) at (9,4) {} edge (4);
        \node [vertex] (6) at (9,3) {} edge (5);
        \node [vertex] (7) at (9,2) {} edge (6);
        \node [vertex] (10) at (10,4) {} edge (6) edge (5);
        \node [vertex,fill=blue!30] (11) at (10,5) {} edge (10) edge (5) edge (4);
        \node[vertex,fill=blue!30] (12) at (10,6) {} edge (11) edge (4)  edge (3);
        \node [vertex] (13) at (10,7) {} edge (12) edge (3) edge (2);
        \node [vertex] (14) at (10,8) {} edge (13) edge (2) edge (1);
        \node [vertex] (15) at (10,3) {} edge (6) edge (10) edge (7);
        \draw[orange,line width=2pt] (4) -- (10,5.5);
        \node at (9.5,2) {$\cdots$};
        \node at (9.5,8.5) {$\cdots$};
        \node at (9.5,1.25) {$=$};
        \begin{scope}[shift={(0,-8)}]
            \node [vertex] (1) at (9,8) {};
            \node [vertex] (2) at (9,7) {} edge (1);
            \node [vertex,fill=blue!30] (3) at (9,6) {} edge (2);
            \node [vertex,draw=orange!120,line width=1pt,fill=blue] (4) at (9,5) {} edge (3);
            \node [vertex,fill=blue!30] (5) at (9,4) {} edge (4);
            \node [vertex] (6) at (9,3) {} edge (5);
            \node [vertex] (7) at (9,2) {} edge (6);
            \node [vertex] (10) at (10,3.5) {} edge (6) edge (5);
            \node [vertex,fill=blue!30] (11) at (10,4.5) {} edge (10) edge (5) edge (4);
            \node[vertex,fill=blue!30] (12) at (10,5.5) {} edge (11) edge (4)  edge (3);
            \node [vertex] (13) at (10,6.5) {} edge (12) edge (3) edge (2);
            \node [vertex] (14) at (10,7.5) {} edge (13) edge (2) edge (1);
            \node [vertex] (15) at (10,2.5) {} edge (6) edge (10) edge (7);
            \draw[orange,line width=2pt] (4) -- (10,5);
            \node at (9.5,1.5) {$\cdots$};
            \node at (9.5,8.5) {$\cdots$};
        \end{scope}
        \end{tikzpicture}
        \caption{First player move when $\oneslant{2}$ when $n$ is even, and where to split the graph.}
        \label{fig:P2_even_n}
    \end{figure}
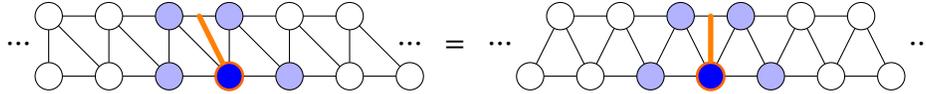
    
    Now in each situation Left can now use \cref{thm:symmetry} to copy Rights moves. This means that Left will play last and win, making \snort on $\basegraph{2}$ and its variation a first player win. 
\end{proof}

\section{\snort played on $\basegraph{3}$ and its variants}

There are a total of seven different variations of $\basegraph{3}$ that we will consider in this section. We will give the winning strategy for the first player when playing \snort on each type of graph in a separate lemma, with the strategy in each case split depending on the parity of $n$.

\begin{lemma}\label{lem:basegraph}
    When playing \snort on empty graphs of the form $\basegraph{3}$, the outcome class is a first player win.
\end{lemma}
\begin{proof}  
    When $n=1$, the first player wins by moving in the central vertex, thereby preventing their opponent from moving.
    
    For the case when $n\geq2$, assume without loss of generality that Left goes first.
    
    First let us consider \snort played on $\basegraph{3}$ when $n\geq 3$ is odd, as shown in \cref{fig:P3_odd_/same}. 
    The optimal first move for Left in this situation is to go in $\left(\lceil n/2\rceil,2\right)$. This allows for Left to think of the graph as split into two halves, that are symmetric aside from some blue shading, by deleting the vertices
    \[\left(\left\lceil\frac{n}{2}\right\rceil,1\right),\; \left(\left\lceil\frac{n}{2}\right\rceil,2\right), \text{ and }\left(\left\lceil\frac{n}{2}\right\rceil,3\right).\]
    
    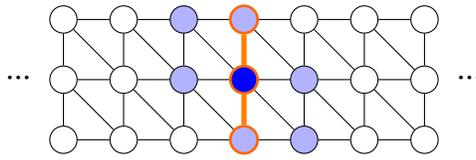
\begin{figure}[!ht]
    \centering
    \begin{tikzpicture}[scale=.4, vertex/.style={circle, draw, minimum size=2.5mm, font=\scriptsize}]
        \node[vertex] (1) at (0,4) {};
        \node[vertex] (2) at (2,4) {} edge (1);
        \node[vertex,fill=blue!30] (3) at (4,4) {} edge (2);
        \node[vertex,draw=orange!120,line width=1pt,fill=blue!30] (4) at (6,4) {} edge (3);
        \node[vertex] (5) at (8,4) {} edge (4);
        \node[vertex] (6) at (10,4) {} edge (5);
        \node[vertex] (7) at (12,4) {} edge (6);
        \node[vertex] (8) at (0,2) {} edge (1);
        \node[vertex] (9) at (2,2) {} edge (8) edge (1) edge (2);
        \node[vertex,fill=blue!30] (10) at (4,2) {} edge (9) edge (3) edge (2);
        \node[vertex,draw=orange!120,line width=1pt,fill=blue] (11) at (6,2) {} edge (10) edge[orange,line width=2pt] (4) edge (3);
        \node[vertex,fill=blue!30] (12) at (8,2) {} edge (11) edge (5) edge (4);
        \node[vertex] (13) at (10,2) {} edge (12) edge (6) edge (5);
        \node[vertex] (14) at (12,2) {} edge (13) edge (7) edge (6);
        \node[vertex] (15) at (0,0) {} edge (8);
        \node[vertex] (16) at (2,0) {} edge (15) edge (8) edge (9);
        \node[vertex] (17) at (4,0) {} edge (16) edge (10) edge (9);
        \node[vertex,draw=orange!120,line width=1pt,fill=blue!30] (18) at (6,0) {} edge (17) edge[orange,line width=2pt] (11) edge (10);
        \node[vertex,fill=blue!30] (19) at (8,0) {} edge (18) edge (12) edge (11);
        \node[vertex] (20) at (10,0) {} edge (19) edge (13) edge (12);
        \node[vertex] (21) at (12,0) {} edge (20) edge (14) edge (13);
        \node at (-1.5,2) {$\cdots$};
        \node at (13.5,2) {$\cdots$};
    \end{tikzpicture}
    \caption{First player move on $\basegraph{3}$ when $n$ is odd, and the line of symmetry}
    \label{fig:P3_odd_/same}
    \end{figure}
    
    Then let us consider \snort played on $\basegraph{3}$ when $n\geq 2$ is even, as shown in \cref{fig:P3_even/same}.
    The optimal first move for Left in this situation is to go in $(n/2+1,2)$. This allows for Left to think of the graph as split into two halves that are rotational symmetric aside from some shading by deleting the vertices:
    \begin{align*}
        \left(\frac{n}{2},1\right),\left(\frac{n}{2},2\right), \left(\frac{n}{2}+1,2\right), \text{ and }\left(\frac{n}{2}+1,3\right).
    \end{align*}
    
    \begin{figure}[!ht]
    \centering
    \begin{tikzpicture}[scale=.4, vertex/.style={circle, draw, minimum size=2.5mm, font=\scriptsize}]
        
        \node[vertex] (1) at (0,4) {};
        \node[vertex] (2) at (2,4) {} edge (1);
        \node[vertex,draw=orange!120,line width=1pt,fill=blue!30] (3) at (4,4) {} edge (2);
        \node[vertex,fill=blue!30] (4) at (6,4) {} edge (3);
        \node[vertex] (5) at (8,4) {} edge (4);
        \node[vertex] (6) at (10,4) {} edge (5);
        
        \node[vertex] (8) at (0,2) {} edge (1);
        \node[vertex] (9) at (2,2) {} edge (8) edge (1) edge (2);
        \node[vertex,draw=orange!120,line width=1pt,fill=blue!30] (10) at (4,2) {} edge (9) edge[orange,line width=2pt] (3) edge (2);
        \node[vertex,draw=orange!120,line width=1pt,fill=blue] (11) at (6,2) {} edge[orange,line width=2pt] (10) edge (4) edge[orange,line width=2pt] (3);
        \node[vertex,fill=blue!30] (12) at (8,2) {} edge (11) edge (5) edge (4);
        \node[vertex] (13) at (10,2) {} edge (12) edge (6) edge (5);
        
        \node[vertex] (15) at (0,0) {} edge (8);
        \node[vertex] (16) at (2,0) {} edge (15) edge (8) edge (9);
        \node[vertex] (17) at (4,0) {} edge (16) edge (10) edge (9);
        \node[vertex,draw=orange!120,line width=1pt,fill=blue!30] (18) at (6,0) {} edge (17) edge[orange,line width=2pt] (11) edge[orange,line width=2pt] (10);
        \node[vertex,fill=blue!30] (19) at (8,0) {} edge (18) edge (12) edge (11);
        \node[vertex] (20) at (10,0) {} edge (19) edge (13) edge (12);
        \node at (-1,2) {$\cdots$};
        \node at (11,2) {$\cdots$};
        \node at (12,2) {$=$};
        \begin{scope}[shift={(14,0)}]
            \node at (-1,2) {$\cdots$};
            \node at (11,2) {$\cdots$};
            
            \node[vertex] (1) at (1,4) {};
            \node[vertex] (2) at (3,4) {} edge (1);
            \node[vertex,draw=orange!120,line width=1pt,fill=blue!30] (3) at (5,4) {} edge (2);
            \node[vertex,fill=blue!30] (4) at (7,4) {} edge (3);
            \node[vertex] (5) at (9,4) {} edge (4);
            \node[vertex] (6) at (11,4) {} edge (5);
            
            \node[vertex] (8) at (0,2) {} edge (1);
            \node[vertex] (9) at (2,2) {} edge (8) edge (1) edge (2);
            \node[vertex,draw=orange!120,line width=1pt,fill=blue!30] (10) at (4,2) {} edge (9) edge[orange,line width=2pt] (3) edge (2);
            \node[vertex,draw=orange!120,line width=1pt,fill=blue] (11) at (6,2) {} edge[orange,line width=2pt] (10) edge (4) edge[orange,line width=2pt] (3);
            \node[vertex,fill=blue!30] (12) at (8,2) {} edge (11) edge (5) edge (4);
            \node[vertex] (13) at (10,2) {} edge (12) edge (6) edge (5);
            
            \node[vertex] (15) at (-1,0) {} edge (8);
            \node[vertex] (16) at (1,0) {} edge (15) edge (8) edge (9);              
            \node[vertex] (17) at (3,0) {} edge (16) edge (10) edge (9);
            \node[vertex,draw=orange!120,line width=1pt,fill=blue!30] (18) at (5,0) {} edge (17) edge[orange,line width=2pt] (11) edge[orange,line width=2pt] (10);
            \node[vertex,fill=blue!30] (19) at (7,0) {} edge (18) edge (12) edge (11);
            \node[vertex] (20) at (9,0) {} edge (19) edge (13) edge (12);
        \end{scope}
    \end{tikzpicture}
    \caption{First player move on $\basegraph{3}$ when $n$ is even, and its line of symmetry}
    \label{fig:P3_even/same}
    \end{figure}

    Thus Left can consider the graph after her first move as being split into two symmetric halves, with some blue shading. Using a copycat strategy, Left will then win by \cref{thm:symmetry}.
\end{proof}

\vspace{5pt}

\begin{lemma}\label{lem:leftaddoneboth}
    When playing \snort on empty graphs of the form $\leftaddoneboth{3}$, the outcome class is a first player win.
\end{lemma}
\begin{proof} 
    The two cases when $n\leq 2$ can be checked by hand or using software such as CGSuite. 
    
    Now assume $n\geq3$ and assume without loss of generality that Left moves first.
    
    First let us consider \snort played on $\leftaddoneboth{3}$ when $n\geq 4$ is even, as shown in \cref{fig:P3_odd/diffrent}.
    The optimal first move for Left in this situation is to go in $\left(n/2+1,2\right)$. This allows for Left to think of the graph as split into two halves that are equal aside from some shading by deleting the vertices:
    \begin{align*}
        \left(\frac{n}{2},1\right), \left(\frac{n}{2}+1,2\right), \text{ and }\left(\frac{n}{2}+1,3\right).
    \end{align*}

    \begin{figure}[!ht]
    \centering
    \begin{tikzpicture}[scale=.4, vertex/.style={circle, draw, minimum size=2.5mm, font=\scriptsize}]
        
        \node[vertex] (1) at (0,4) {};
        \node[vertex] (2) at (2,4) {} edge (1);
        \node[vertex,draw=orange!120,line width=1pt,fill=blue!30] (3) at (4,4) {} edge (2);
        \node[vertex,fill=blue!30] (4) at (6,4) {} edge (3);
        \node[vertex] (5) at (8,4) {} edge (4);
        \node[vertex] (6) at (10,4) {} edge (5);
        \node[vertex] (1a) at (-2,4) {} edge (1);
        
        \node[vertex] (8) at (0,2) {} edge (1) edge (1a);
        \node[vertex] (9) at (2,2) {} edge (8) edge (1) edge (2);
        \node[vertex,fill=blue!30] (10) at (4,2) {} edge (9) edge (3) edge (2);
        \node[vertex,draw=orange!120,line width=1pt,fill=blue] (11) at (6,2) {} edge (10) edge (4) edge[orange,line width=2pt] (3);
        \node[vertex,fill=blue!30] (12) at (8,2) {} edge (11) edge (5) edge (4);
        \node[vertex] (13) at (10,2) {} edge (12) edge (6) edge (5);
        \node[vertex] (14) at (12,2) {} edge (13) edge (6);
        
        \node[vertex] (15) at (0,0) {} edge (8);
        \node[vertex] (16) at (2,0) {} edge (15) edge (8) edge (9);
        \node[vertex] (17) at (4,0) {} edge (16) edge (10) edge (9);
        \node[vertex,draw=orange!120,line width=1pt,fill=blue!30] (18) at (6,0) {} edge (17) edge[orange,line width=2pt] (11) edge (10);
        \node[vertex,fill=blue!30] (19) at (8,0) {} edge (18) edge (12) edge (11);
        \node[vertex] (20) at (10,0) {} edge (19) edge (13) edge (12);
        \node[vertex] (21) at (12,0) {} edge (20) edge (14) edge (13);
        \node at (-2.5,2) {$\cdots$};
        \node at (13,2) {$\cdots$};
        \node at (14,2) {$=$};
        \begin{scope}[shift={(16.5,0)}]
            \node at (-1.5,2) {$\cdots$};
            \node at (13,2) {$\cdots$};
            \node[vertex] (1) at (1,4) {};
            \node[vertex] (2) at (3,4) {} edge (1);
            \node[vertex,draw=orange!120,line width=1pt,fill=blue!30] (3) at (5,4) {} edge (2);
            \node[vertex,fill=blue!30] (4) at (7,4) {} edge (3);
            \node[vertex] (5) at (9,4) {} edge (4);
            \node[vertex] (6) at (11,4) {} edge (5);
            \node[vertex] (1a) at (-1,4) {} edge (1);
            
            \node[vertex] (8) at (0,2) {} edge (1) edge (1a);
            \node[vertex] (9) at (2,2) {} edge (8) edge (1) edge (2);
            \node[vertex,fill=blue!30] (10) at (4,2) {} edge (9) edge (3) edge (2);
            \node[vertex,draw=orange!120,line width=1pt,fill=blue] (11) at (6,2) {} edge (10) edge (4) edge[orange,line width=2pt] (3);
            \node[vertex,fill=blue!30] (12) at (8,2) {} edge (11) edge (5) edge (4);
            \node[vertex] (13) at (10,2) {} edge (12) edge (6) edge (5);
            \node[vertex] (14) at (12,2) {} edge (13) edge (6);
            
            \node[vertex] (15) at (-1,0) {} edge (8);
            \node[vertex] (16) at (1,0) {} edge (15) edge (8) edge (9);
            \node[vertex] (17) at (3,0) {} edge (16) edge (10) edge (9);
            \node[vertex,draw=orange!120,line width=1pt,fill=blue!30] (18) at (5,0) {} edge (17) edge[orange,line width=2pt] (11) edge (10);
            \node[vertex,fill=blue!30] (19) at (7,0) {} edge (18) edge (12) edge (11);
            \node[vertex] (20) at (9,0) {} edge (19) edge (13) edge (12);
            \node[vertex] (21) at (11,0) {} edge (20) edge (14) edge (13);
        \end{scope}
    \end{tikzpicture}
    \caption{First player move on $\leftaddoneboth{3}$ when $n$ is even, and its line of symmetry}
    \label{fig:P3_odd/diffrent}
    \end{figure}

    Then let us consider \snort played on $\leftaddoneboth{3}$ when $n\geq 3$ is odd, as shown in \cref{fig:P3_even/diffrent}.
    The optimal first move for Left in this situation is to go in $\left(\lceil n/2\rceil,2\right)$. 
    This allows for Left to think of the graph as split into two halves that are equal aside from some shading by deleting the vertices:
    \begin{align*}
        \left(\left\lceil\frac{n}{2}\right\rceil-1,1\right), \left(\left\lceil\frac{n}{2}\right\rceil,1\right), \left(\left\lceil\frac{n}{2}\right\rceil,2\right), \left(\left\lceil\frac{n}{2}\right\rceil+1,2\right), \left(\left\lceil\frac{n}{2}\right\rceil,3\right), \text{ and } \left(\left\lceil\frac{n}{2}\right\rceil+1,3\right).
    \end{align*}
   
    \begin{figure}[!ht]
    \centering
    \begin{tikzpicture}[scale=.4, vertex/.style={circle, draw, minimum size=2.5mm, font=\scriptsize}]
        
        \node[vertex] (1) at (0,4) {};
        \node[vertex] (2) at (2,4) {} edge (1);
        \node[vertex,draw=orange!120,line width=1pt,fill=blue!30] (3) at (4,4) {} edge (2);
        \node[vertex,draw=orange!120,line width=1pt,fill=blue!30] (4) at (6,4) {} edge[orange,line width=2pt] (3);
        \node[vertex] (5) at (8,4) {} edge (4);
        \node[vertex] (6) at (10,4) {} edge (5);
        
        \node[vertex] (9) at (2,2) {} edge (1) edge (2);
        \node[vertex,fill=blue!30] (10) at (4,2) {} edge (9) edge (3) edge (2);
        \node[vertex,draw=orange!120,line width=1pt,fill=blue] (11) at (6,2) {} edge (10) edge[orange,line width=2pt] (4) edge[orange,line width=2pt] (3);
        \node[vertex,draw=orange!120,line width=1pt,fill=blue!30] (12) at (8,2) {} edge[orange,line width=2pt] (11) edge (5) edge[orange,line width=2pt] (4);
        \node[vertex] (13) at (10,2) {} edge (12) edge (6) edge (5);
        \node[vertex] (14) at (12,2) {} edge (13) edge (6);
          
        \node[vertex] (16) at (2,0) {} edge (9);
        \node[vertex] (17) at (4,0) {} edge (16) edge (10) edge (9);
        \node[vertex,draw=orange!120,line width=1pt,fill=blue!30] (18) at (6,0) {} edge (17) edge[orange,line width=2pt] (11) edge (10);
        \node[vertex,draw=orange!120,line width=1pt,fill=blue!30] (19) at (8,0) {} edge[orange,line width=2pt] (18) edge[orange,line width=2pt] (12) edge[orange,line width=2pt] (11);
        \node[vertex] (20) at (10,0) {} edge (19) edge (13) edge (12);
        \node[vertex] (21) at (12,0) {} edge (20) edge (14) edge (13);
        \node at (-.5,2) {$\cdots$};
        \node at (13,2) {$\cdots$};
        \node at (14,2) {$=$};
        \begin{scope}[shift={(14.5,0)}]
            \node at (.5,2) {$\cdots$};
            \node at (13,2) {$\cdots$};
            \node[vertex] (1) at (1,4) {};
            \node[vertex] (2) at (3,4) {} edge (1);
            \node[vertex,draw=orange!120,line width=1pt,fill=blue!30] (3) at (5,4) {} edge (2);
            \node[vertex,draw=orange!120,line width=1pt,fill=blue!30] (4) at (7,4) {} edge[orange,line width=2pt] (3);
            \node[vertex] (5) at (9,4) {} edge (4);
            \node[vertex] (6) at (11,4) {} edge (5);
            
            \node[vertex] (9) at (2,2) {} edge (1) edge (2);
            \node[vertex,fill=blue!30] (10) at (4,2) {} edge (9) edge (3) edge (2);
            \node[vertex,draw=orange!120,line width=1pt,fill=blue] (11) at (6,2) {} edge (10) edge[orange,line width=2pt] (4) edge[orange,line width=2pt] (3);
            \node[vertex,draw=orange!120,line width=1pt,fill=blue!30] (12) at (8,2) {} edge[orange,line width=2pt] (11) edge (5) edge[orange,line width=2pt] (4);
            \node[vertex] (13) at (10,2) {} edge (12) edge (6) edge (5);
            \node[vertex] (14) at (12,2) {} edge (13) edge (6);
              
            \node[vertex] (16) at (1,0) {} edge (9);
            \node[vertex] (17) at (3,0) {} edge (16) edge (10) edge (9);
            \node[vertex,draw=orange!120,line width=1pt,fill=blue!30] (18) at (5,0) {} edge (17) edge[orange,line width=2pt] (11) edge (10);
            \node[vertex,draw=orange!120,line width=1pt,fill=blue!30] (19) at (7,0) {} edge[orange,line width=2pt] (18) edge[orange,line width=2pt] (12) edge[orange,line width=2pt] (11);
            \node[vertex] (20) at (9,0) {} edge (19) edge (13) edge (12);
            \node[vertex] (21) at (11,0) {} edge (20) edge (14) edge (13);   
        \end{scope}
    \end{tikzpicture}
    \caption{First player move on $\leftaddoneboth{3}$ when $n$ is odd, and its line of symmetry}
    \label{fig:P3_even/diffrent}
    \end{figure}

    In either case, using \cref{thm:symmetry}, Left will win. Thus this game is a first player win.
\end{proof}

\vspace{5pt}
    
\begin{lemma}\label{lem:bothaddone}
    When playing \snort on empty graphs of the form $\bothaddone{3}$, the outcome class is a first player win. 
\end{lemma}
\begin{proof} 
    When $n=1$, the first player will play in the central vertex and prevent their opponent from moving.
    
    When $n\geq2$, assume without loss of generality that Left moves first.

    First let us consider \snort played on the empty graph $\bothaddone{3}$ when $n\geq 3$ is odd, as shown in \cref{fig:P3_cen_odd}.
    The optimal first move for Left in this situation is to move in $\left(\lceil n/2\rceil,2\right)$. Left will then think of the graph as split by deleting the vertex $\left(\lceil n/2\rceil,2\right)$ and the edges $\left(\left(\lceil n/2\rceil-1,1\right),\left(\lceil n/2\rceil,1\right)\right)$ and $\left(\left(\lceil n/2\rceil,3\right),\left(\lceil n/2+1\rceil,3\right)\right)$. 
    
    \begin{figure}[!ht]
    \centering
    \begin{tikzpicture}[scale=.4, vertex/.style={circle, draw, minimum size=2.5mm, font=\scriptsize}]
        
        \node[vertex] (1) at (0,4) {};
        \node[vertex] (2) at (2,4) {} edge (1);
        \node[vertex,fill=blue!30] (3) at (4,4) {} edge (2);
        \node[vertex,fill=blue!30] (4) at (6,4) {} edge (3);
        \node[vertex] (5) at (8,4) {} edge (4);
        \node[vertex] (6) at (10,4) {} edge (5);
        
        \node[vertex] (9) at (2,2) {} edge (1) edge (2);
        \node[vertex,fill=blue!30] (10) at (4,2) {} edge (9) edge (3) edge (2);
        \node[vertex,draw=orange!120,line width=1pt,fill=blue] (11) at (6,2) {} edge (10) edge (4) edge (3);
        \node[vertex,fill=blue!30] (12) at (8,2) {} edge (11) edge (5) edge (4);
        \node[vertex] (13) at (10,2) {} edge (12) edge (6) edge (5);
        
        \node[vertex] (16) at (2,0) {} edge (9);
        \node[vertex] (17) at (4,0) {} edge (16) edge (10) edge (9);
        \node[vertex,fill=blue!30] (18) at (6,0) {} edge (17) edge (11) edge (10);
        \node[vertex,fill=blue!30] (19) at (8,0) {} edge (18) edge (12) edge (11);
        \node[vertex] (20) at (10,0) {} edge (19) edge (13) edge (12);
        \node[vertex] (21) at (12,0) {} edge (20) edge (13);
        \draw[orange,line width=2pt] (5,4) -- (11);
        \draw[orange,line width=2pt] (7,0) -- (11);
        \node at (-.5,2) {$\cdots$};
        \node at (12.5,2) {$\cdots$};
        \node at (13.5,2) {$=$};
        \begin{scope}[shift={(14,0)}]
            \node at (.5,2) {$\cdots$};
            \node at (11.5,2) {$\cdots$};
            \node[vertex] (1) at (1,4) {};
            \node[vertex] (2) at (3,4) {} edge (1);
            \node[vertex,fill=blue!30] (3) at (5,4) {} edge (2);
            \node[vertex,fill=blue!30] (4) at (7,4) {} edge (3);
            \node[vertex] (5) at (9,4) {} edge (4);
            \node[vertex] (6) at (11,4) {} edge (5);
            
            \node[vertex] (9) at (2,2) {} edge (1) edge (2);
            \node[vertex,fill=blue!30] (10) at (4,2) {} edge (9) edge (3) edge (2);
            \node[vertex,draw=orange!120,line width=1pt,fill=blue] (11) at (6,2) {} edge (10) edge (4) edge (3);
            \node[vertex,fill=blue!30] (12) at (8,2) {} edge (11) edge (5) edge (4);                
            \node[vertex] (13) at (10,2) {} edge (12) edge (6) edge (5);
            
            \node[vertex] (16) at (1,0) {} edge (9);
            \node[vertex] (17) at (3,0) {} edge (16) edge (10) edge (9);
            \node[vertex,fill=blue!30] (18) at (5,0) {} edge (17) edge (11) edge (10);
            \node[vertex,fill=blue!30] (19) at (7,0) {} edge (18) edge (12) edge (11);
            \node[vertex] (20) at (9,0) {} edge (19) edge (13) edge (12);
            \node[vertex] (21) at (11,0) {} edge (20) edge (13);
            \draw[orange,line width=2pt] (6,4) -- (11);
            \draw[orange,line width=2pt] (6,0) -- (11);
        \end{scope}
    \end{tikzpicture}
    \caption{First player move on $\bothaddone{3}$ when $n$ is odd, and its line of symmetry}
    \label{fig:P3_cen_odd}
    \end{figure}
    
    \vspace{5pt}
    
    Then let us consider the case when $n\geq 2$ is even, as shown in figure \cref{fig:P3_cen_even}.
    The optimal first move for Left in this situation is to move in $\left(n/2+1,2\right)$. Left will then think of the graph as split by deleting the vertices
    \begin{align*}
        \left(\frac{n}{2},1\right), \left(\frac{n}{2},2\right), \left(\frac{n}{2}+1,2\right), \text{ and }\left(\frac{n}{2}+1,3\right).
    \end{align*}

    \begin{figure}[!ht]
    \centering
    \begin{tikzpicture}[scale=.4, vertex/.style={circle, draw, minimum size=2.5mm, font=\scriptsize}]
        
        \node[vertex] (1) at (0,4) {};
        \node[vertex] (2) at (2,4) {} edge (1);
        \node[vertex,draw=orange!120,line width=1pt,fill=blue!30] (3) at (4,4) {} edge (2);
        \node[vertex,fill=blue!30] (4) at (6,4) {} edge (3);
        \node[vertex] (5) at (8,4) {} edge (4);
        \node[vertex] (6) at (10,4) {} edge (5);
        \node[vertex] (1a) at (-2,4) {} edge (1);
        
        \node[vertex] (8) at (0,2) {} edge (1) edge (1a);
        \node[vertex] (9) at (2,2) {} edge (8) edge (1) edge (2);
        \node[vertex,draw=orange!120,line width=1pt,fill=blue!30] (10) at (4,2) {} edge (9) edge[orange,line width=2pt] (3) edge (2);
        \node[vertex,draw=orange!120,line width=1pt,fill=blue] (11) at (6,2) {} edge[orange,line width=2pt] (10) edge (4) edge[orange,line width=2pt] (3);
        \node[vertex,fill=blue!30] (12) at (8,2) {} edge (11) edge (5) edge (4);
        \node[vertex] (13) at (10,2) {} edge (12) edge (6) edge (5);
        
        \node[vertex] (15) at (0,0) {} edge (8);
        \node[vertex] (16) at (2,0) {} edge (15) edge (8) edge (9);
        \node[vertex] (17) at (4,0) {} edge (16) edge (10) edge (9);
        \node[vertex,draw=orange!120,line width=1pt,fill=blue!30] (18) at (6,0) {} edge (17) edge[orange,line width=2pt] (11) edge[orange,line width=2pt] (10);
        \node[vertex,fill=blue!30] (19) at (8,0) {} edge (18) edge (12) edge (11);
        \node[vertex] (20) at (10,0) {} edge (19) edge (13) edge (12);
        \node[vertex] (21) at (12,0) {} edge (20) edge (13);
        \node at (-2.5,2) {$\cdots$};
        \node at (12.5,2) {$\cdots$};
        \node at (13.5,2) {$=$};
        \begin{scope}[shift={(16,0)}]
            \node at (-1.5,2) {$\cdots$};
            \node at (12,2) {$\cdots$};
            
            \node[vertex] (1) at (1,4) {};
            \node[vertex] (2) at (3,4) {} edge (1);
            \node[vertex,draw=orange!120,line width=1pt,fill=blue!30] (3) at (5,4) {} edge (2);
            \node[vertex,fill=blue!30] (4) at (7,4) {} edge (3);
            \node[vertex] (5) at (9,4) {} edge (4);
            \node[vertex] (6) at (11,4) {} edge (5);
            \node[vertex] (1a) at (-1,4) {} edge (1);
            
            \node[vertex] (8) at (0,2) {} edge (1) edge (1a);
            \node[vertex] (9) at (2,2) {} edge (8) edge (1) edge (2);
            \node[vertex,draw=orange!120,line width=1pt,fill=blue!30] (10) at (4,2) {} edge (9) edge[orange,line width=2pt] (3) edge (2);
            \node[vertex,draw=orange!120,line width=1pt,fill=blue] (11) at (6,2) {} edge[orange,line width=2pt] (10) edge (4) edge[orange,line width=2pt] (3);
            \node[vertex,fill=blue!30] (12) at (8,2) {} edge (11) edge (5) edge (4);
            \node[vertex] (13) at (10,2) {} edge (12) edge (6) edge (5);
            
            \node[vertex] (15) at (-1,0) {} edge (8);
            \node[vertex] (16) at (1,0) {} edge (15) edge (8) edge (9);
            \node[vertex] (17) at (3,0) {} edge (16) edge (10) edge (9);
            \node[vertex,draw=orange!120,line width=1pt,fill=blue!30] (18) at (5,0) {} edge (17) edge[orange,line width=2pt] (11) edge[orange,line width=2pt] (10);
            \node[vertex,fill=blue!30] (19) at (7,0) {} edge (18) edge (12) edge (11);
            \node[vertex] (20) at (9,0) {} edge (19) edge (13) edge (12);
            \node[vertex] (21) at (11,0) {} edge (20) edge (13);
        \end{scope}
    \end{tikzpicture}
    \caption{First player move on $\bothaddone{3}$ when $n$ is even, and its line of symmetry}
    \label{fig:P3_cen_even}
    \end{figure}

    Left thinking of splitting the graphs in this way allows for her to use \cref{thm:symmetry}, allowing Left to copy Right's moves and therefore playing last and winning. 
\end{proof}

\vspace{5pt}
    
\begin{lemma}\label{lem:bothminusone}
    When playing \snort on empty graphs of the form $\bothminusone{3}$, the outcome class is a first player win.
\end{lemma}
\begin{proof}
    When $n=1$, the first player will play in the central vertex and prevent their opponent from moving. 
    
    When $n\geq2$, assume without loss of generality that Left goes first.

    When $n$ is odd, the optimal first move for Left is in $\left(\lceil n/2\rceil,2\right)$. This allows for Left to think of the graph as if it were split into two equal graphs aside from some shading by the one hand tied principle. This can be done by deleting the vertex $\left(\lceil n/2\rceil,2\right)$ and the edges $\left(\left(\lceil n/2\rceil-1,1\right),\left(\lceil n/2\rceil,1\right)\right)$, and $\left(\left(\lceil n/2\rceil,3\right),\left(\lceil n/2+1\rceil,3\right)\right)$.    
    
    \begin{figure}[!ht]
    \centering
    \begin{tikzpicture}[scale=.4, vertex/.style={circle, draw, minimum size=2.5mm, font=\scriptsize}]
        
        \node[vertex] (1) at (0,4) {};
        \node[vertex] (2) at (2,4) {} edge (1);  
        \node[vertex,fill=blue!30] (3) at (4,4) {} edge (2);
        \node[vertex,fill=blue!30] (4) at (6,4) {} edge (3);
        \node[vertex] (5) at (8,4) {} edge (4);
        \node[vertex] (6) at (10,4) {} edge (5);
        
        \node[vertex] (8) at (0,2) {} edge (1);
        \node[vertex] (9) at (2,2) {} edge (1) edge (2) edge (8);
        \node[vertex,fill=blue!30] (10) at (4,2) {} edge (9) edge (3) edge (2);
        \node[vertex,draw=orange!120,line width=1pt,fill=blue] (11) at (6,2) {} edge (10) edge (4) edge (3);
        \node[vertex,fill=blue!30] (12) at (8,2) {} edge (11) edge (5) edge (4);
        \node[vertex] (13) at (10,2) {} edge (12) edge (6) edge (5);
        \node[vertex] (14) at (12,2) {} edge (13) edge (6);
        
        \node[vertex] (16) at (2,0) {} edge (9) edge (8);
        \node[vertex] (17) at (4,0) {} edge (16) edge (10) edge (9);
        \node[vertex,fill=blue!30] (18) at (6,0) {} edge (17) edge (11) edge (10);
        \node[vertex,fill=blue!30] (19) at (8,0) {} edge (18) edge (12) edge (11);
        \node[vertex] (20) at (10,0) {} edge (19) edge (13) edge (12);
        \node[vertex] (21) at (12,0) {} edge (20) edge (13) edge (14);
        \draw[orange,line width=2pt] (5,4) -- (11);
        \draw[orange,line width=2pt] (7,0) -- (11);
        \node at (-1,2) {$\cdots$};
        \node at (13,2) {$\cdots$};
        \node at (14,2) {$=$};
        \begin{scope}[shift={(16,0)}]
            \node at (-1,2) {$\cdots$};
            \node at (13,2) {$\cdots$}; 
            \node[vertex] (1) at (1,4) {};
            \node[vertex] (2) at (3,4) {} edge (1);
            \node[vertex,fill=blue!30] (3) at (5,4) {} edge (2);
            \node[vertex,fill=blue!30] (4) at (7,4) {} edge (3);
            \node[vertex] (5) at (9,4) {} edge (4);
            \node[vertex] (6) at (11,4) {} edge (5);
            
            \node[vertex] (8) at (0,2) {} edge (1);
            \node[vertex] (9) at (2,2) {} edge (1) edge (2) edge (8);
            \node[vertex,fill=blue!30] (10) at (4,2) {} edge (9) edge (3) edge (2);
            \node[vertex,draw=orange!120,line width=1pt,fill=blue] (11) at (6,2) {} edge (10) edge (4) edge (3);
            \node[vertex,fill=blue!30] (12) at (8,2) {} edge (11) edge (5) edge (4);
            \node[vertex] (13) at (10,2) {} edge (12) edge (6) edge (5);
            \node[vertex] (14) at (12,2) {} edge (13) edge (6);
            
            \node[vertex] (16) at (1,0) {} edge (9) edge (8);
            \node[vertex] (17) at (3,0) {} edge (16) edge (10) edge (9);
            \node[vertex,fill=blue!30] (18) at (5,0) {} edge (17) edge (11) edge (10);
            \node[vertex,fill=blue!30] (19) at (7,0) {} edge (18) edge (12) edge (11);
            \node[vertex] (20) at (9,0) {} edge (19) edge (13) edge (12);
            \node[vertex] (21) at (11,0) {} edge (20) edge (13) edge (14);
            \draw[orange,line width=2pt] (6,4) -- (11);
            \draw[orange,line width=2pt] (6,0) -- (11);
        \end{scope}
    \end{tikzpicture}
    \caption{First player move on $\bothminusone{3}$ when $n$ is odd, and its line of symmetry}
    \label{fig:P3_cen_minus_odd}
    \end{figure}\
    
    Then let us consider the game \snort played on the empty graph $\bothminusone{3}$ and $n$ is even, as shown in \cref{fig:P3_cen_minus_even}.
    The optimal first move for Left in this situation is to move in $\left(n/2+1,2\right)$. This allows for Left to think of the graph as if it were split into two equal graphs aside from some shading by the one hand tied principle. This can be done by deleting the vertices:
    \begin{align*}
        \left(\frac{n}{2},1\right), \left(\frac{n}{2},2\right), \left(\frac{n}{2}+1,2\right), \text{ and } \left(\frac{n}{2}+1,3\right).
    \end{align*}
    
    \begin{figure}[!ht]
    \centering
    \begin{tikzpicture}[scale=.4, vertex/.style={circle, draw, minimum size=2.5mm, font=\scriptsize}]
        
        \node[vertex] (1) at (0,4) {};
        \node[vertex] (2) at (2,4) {} edge (1);
        \node[vertex,draw=orange!120,line width=1pt,fill=blue!30] (3) at (4,4) {} edge (2);
        \node[vertex,fill=blue!30] (4) at (6,4) {} edge (3);
        \node[vertex] (5) at (8,4) {} edge (4);
        \node[vertex] (6) at (10,4) {} edge (5);
        \node[vertex] (1a) at (-2,4) {} edge (1);
        
        \node[vertex] (8a) at (-2,2) {} edge (1a);
        \node[vertex] (8) at (0,2) {} edge (1) edge (1a) edge (8a);
        \node[vertex] (9) at (2,2) {} edge (8) edge (1) edge (2);
        \node[vertex,draw=orange!120,line width=1pt,fill=blue!30] (10) at (4,2) {} edge (9) edge[orange,line width=2pt] (3) edge (2);
        \node[vertex,draw=orange!120,line width=1pt,fill=blue] (11) at (6,2) {} edge[orange,line width=2pt] (10) edge[orange,line width=2pt] (3) edge (4);
        \node[vertex,fill=blue!30] (12) at (8,2) {} edge (11) edge (5) edge (4);
        \node[vertex] (13) at (10,2) {} edge (12) edge (6) edge (5);
        \node[vertex] (13a) at (12,2) {} edge (13) edge (6);
        
        \node[vertex] (15) at (0,0) {} edge (8) edge (8a);
        \node[vertex] (16) at (2,0) {} edge (15) edge (8) edge (9);
        \node[vertex] (17) at (4,0) {} edge (16) edge (10) edge (9);
        \node[vertex,draw=orange!120,line width=1pt,fill=blue!30] (18) at (6,0) {} edge (17) edge[orange,line width=2pt] (11) edge[orange,line width=2pt] (10);
        \node[vertex,fill=blue!30] (19) at (8,0) {} edge (18) edge (12) edge (11);
        \node[vertex] (20) at (10,0) {} edge (19) edge (13) edge (12);
        \node[vertex] (21) at (12,0) {} edge (20) edge (13) edge (13a);
        \node at (-3,2) {$\cdots$};
        \node at (13,2) {$\cdots$};
        \node at (14,2) {$=$};
        \begin{scope}[shift={(18,0)}]
            \node at (-3,2) {$\cdots$};
            \node at (13,2) {$\cdots$};  
            \node[vertex] (1) at (1,4) {};
            \node[vertex] (2) at (3,4) {} edge (1);
            \node[vertex,draw=orange!120,line width=1pt,fill=blue!30] (3) at (5,4) {} edge (2);
            \node[vertex,fill=blue!30] (4) at (7,4) {} edge (3);
            \node[vertex] (5) at (9,4) {} edge (4);
            \node[vertex] (6) at (11,4) {} edge (5);
            \node[vertex] (1a) at (-1,4) {} edge (1);
            
            \node[vertex] (8a) at (-2,2) {} edge (1a);
            \node[vertex] (8) at (0,2) {} edge (1) edge (1a) edge (8a);
            \node[vertex] (9) at (2,2) {} edge (8) edge (1) edge (2);
            \node[vertex,draw=orange!120,line width=1pt,fill=blue!30] (10) at (4,2) {} edge (9) edge[orange,line width=2pt] (3) edge (2);
            \node[vertex,draw=orange!120,line width=1pt,fill=blue] (11) at (6,2) {} edge[orange,line width=2pt] (10) edge (4) edge[orange,line width=2pt] (3);
            \node[vertex,fill=blue!30] (12) at (8,2) {} edge (11) edge (5) edge (4);
            \node[vertex] (13) at (10,2) {} edge (12) edge (6) edge (5);
            \node[vertex] (13a) at (12,2) {} edge (13) edge (6);
            
            \node[vertex] (15) at (-1,0) {} edge (8) edge (8a);
            \node[vertex] (16) at (1,0) {} edge (15) edge (8) edge (9);
            \node[vertex] (17) at (3,0) {} edge (16) edge (10) edge (9);
            \node[vertex,draw=orange!120,line width=1pt,fill=blue!30] (18) at (5,0) {} edge (17) edge[orange,line width=2pt] (11) edge[orange,line width=2pt] (10);
            \node[vertex,fill=blue!30] (19) at (7,0) {} edge (18) edge (12) edge (11);
            \node[vertex] (20) at (9,0) {} edge (19) edge (13) edge (12);
            \node[vertex] (21) at (11,0) {} edge (20) edge (13) edge (13a);
        \end{scope}
    \end{tikzpicture}
    \caption{First player move on $\bothminusone{3}$ when $n$ is even, and its line of symmetry}
    \label{fig:P3_cen_minus_even}
    \end{figure}

    By moving first in these positions it will allow for Left to think as if the graph is split into two graphs that are rotational symmetric aside from some shading, so using \cref{thm:symmetry}, Left will play last and win making it a first player win.
\end{proof}

\vspace{5pt}

\begin{lemma}\label{lem:oneslant}
    When playing \snort on empty graphs of the form $\oneslant{3}$, the outcome class is a  first player win.
\end{lemma}
\begin{proof}
    When $n=1$, the first player will play in the vertex $(1,2)$. The opponent then has only one move (in $R_3'$), which the first player can respond to in $(1,1)$ and wins.
    
    For graphs where $n\geq2$, assume without loss of generality that Left plays first.

    First let us consider the game \snort played on the empty graph $\oneslant{3}$ and $n\geq 3$ is odd, as shown in \cref{fig:trap_cen_even}. 
    The optimal first move for Left in this situation is to move in $\left(\lceil n/2\rceil,2\right)$. This allows for Left to think of the graph as if it were split into two isomorphic graphs aside from some shading by the one hand tied principle. This can be done by deleting the vertices
    \begin{align*}
        \left(\frac{n}{2},1\right), \left(\frac{n}{2},2\right), \left(\frac{n}{2}+1,2\right), \text{ and } \left(\frac{n}{2}+1,3\right).
    \end{align*}
    
    \begin{figure}[!ht]
    \centering
    \begin{tikzpicture}[scale=.4, vertex/.style={circle, draw, minimum size=2.5mm, font=\scriptsize}]
        
        \node[vertex] (1) at (0,4) {};
        \node[vertex,fill=blue!30] (2) at (2,4) {} edge (1);
        \node[vertex,draw=orange!120,line width=1pt,fill=blue!30] (3) at (4,4) {} edge (2);
        \node[vertex] (4) at (6,4) {} edge (3);
        \node[vertex] (5) at (8,4) {} edge (4);
        
        \node[vertex] (8) at (0,2) {} edge (1);
        \node[vertex,fill=blue!30] (9) at (2,2) {} edge (8) edge (1) edge (2);
        \node[vertex,draw=orange!120,line width=1pt,fill=blue] (10) at (4,2) {} edge (9) edge[orange,line width=2pt] (3) edge (2);
        \node[vertex,draw=orange!120,line width=1pt,fill=blue!30] (11) at (6,2) {} edge[orange,line width=2pt] (10) edge (4) edge[orange,line width=2pt] (3);
        \node[vertex] (12) at (8,2) {} edge (11) edge (5) edge (4);
        \node[vertex] (13) at (10,2) {} edge (12) edge (5);
        
        \node[vertex] (15) at (0,0) {} edge (8);
        \node[vertex] (16) at (2,0) {} edge (15) edge (8) edge (9);
        \node[vertex,fill=blue!30] (17) at (4,0) {} edge (16) edge (10) edge (9);
        \node[vertex,draw=orange!120,line width=1pt,fill=blue!30] (18) at (6,0) {} edge (17) edge[orange,line width=2pt] (11) edge[orange,line width=2pt] (10);
        \node[vertex] (19) at (8,0) {} edge (18) edge (12) edge (11);
        \node[vertex] (20) at (10,0) {} edge (19) edge (13) edge (12);
        \node[vertex] (21) at (12,0) {} edge (20) edge (13);
        \node at (-1,2) {$\cdots$};
        \node at (12,2) {$\cdots$};
        \node at (13,2) {$=$};
        \begin{scope}[shift={(14.5,0)}]
            \node at (-.5,2) {$\cdots$};
            \node at (12.5,2) {$\cdots$}; 
            \node[vertex] (1) at (2,4) {};
            \node[vertex,fill=blue!30] (2) at (4,4) {} edge (1);
            \node[vertex,draw=orange!120,line width=1pt,fill=blue!30] (3) at (6,4) {} edge (2);
            \node[vertex] (4) at (8,4) {} edge (3);
            \node[vertex] (5) at (10,4) {} edge (4);
            
            \node[vertex] (8) at (1,2) {} edge (1);
            \node[vertex,fill=blue!30] (9) at (3,2) {} edge (8) edge (1) edge (2);
            \node[vertex,draw=orange!120,line width=1pt,fill=blue] (10) at (5,2) {} edge (9) edge[orange,line width=2pt] (3) edge (2);
            \node[vertex,draw=orange!120,line width=1pt,fill=blue!30] (11) at (7,2) {} edge[orange,line width=2pt] (10) edge (4) edge[orange,line width=2pt] (3);
            \node[vertex] (12) at (9,2) {} edge (11) edge (5) edge (4);
            \node[vertex] (13) at (11,2) {} edge (12) edge (5);
            
            \node[vertex] (15) at (0,0) {} edge (8);   
            \node[vertex] (16) at (2,0) {} edge (15) edge (8) edge (9);
            \node[vertex,fill=blue!30] (17) at (4,0) {} edge (16) edge (10) edge (9);
            \node[vertex,draw=orange!120,line width=1pt,fill=blue!30] (18) at (6,0) {} edge (17) edge[orange,line width=2pt] (11) edge[orange,line width=2pt] (10);
            \node[vertex] (19) at (8,0) {} edge (18) edge (12) edge (11);
            \node[vertex] (20) at (10,0) {} edge (19) edge (13) edge (12);
            \node[vertex] (21) at (12,0) {} edge (20) edge (13);
        \end{scope}
    \end{tikzpicture}
    \caption{First player move on $\oneslant{3}$ when $n$ is odd, and its line of symmetry} 
    \label{fig:trap_cen_even}
    \end{figure}

    \vspace{5pt}
    
    Then let us consider the game \snort played on the empty graph $\oneslant{3}$ and $n\geq 2$ is even, as shown in \cref{fig:trap_cen_odd}.
    The optimal first move for Left in this situation is to move in $\left(n/2+1,2\right)$. This allows for Left to think of the graph as if it were split into two isomorphic graphs aside from some shading by the one hand tied principle. This can be done by deleting the vertex $\left(n/2+1,2\right)$ and the edges $\left(\left(n/2,1\right), \left(n/2+1,1\right)\right)$ and $\left(\left(n/2+1,3\right), \left(n/2+2,3\right)\right)$.
    
    \begin{figure}[!ht]
    \centering
    \begin{tikzpicture}[scale=.4, vertex/.style={circle, draw, minimum size=2.5mm, font=\scriptsize}]
        
        \node[vertex] (a) at (-2,4) {};
        \node[vertex] (b) at (-4,4) {} edge (a);
        \node[vertex] (c) at (-2,2) {} edge (a) edge (b);
        \node[vertex] (d) at (-4,2) {} edge (b) edge (c);
        \node[vertex] (e) at (-2,0) {} edge (c) edge (d);
        \node[vertex] (f) at (-4,0) {} edge (d) edge (e);
        
        \node[vertex,fill=blue!30] (1) at (0,4) {} edge (a);
        \node[vertex,fill=blue!30] (2) at (2,4) {} edge (1);
        \node[vertex] (3) at (4,4) {} edge (2);
        \node[vertex] (4) at (6,4) {} edge (3);
        
        \node[vertex,fill=blue!30] (8) at (0,2) {} edge (1) edge (c) edge (a);
        \node[vertex,draw=orange!120,line width=1pt,fill=blue] (9) at (2,2) {} edge (8) edge (1) edge (2);
        \node[vertex,fill=blue!30] (10) at (4,2) {} edge (9) edge (3) edge (2);
        \node[vertex] (11) at (6,2) {} edge (10) edge (4) edge (3);
        \node[vertex] (12) at (8,2) {} edge (11) edge (4);
        
        \node[vertex] (15) at (0,0) {} edge (8) edge (e) edge (c);
        \node[vertex,fill=blue!30] (16) at (2,0) {} edge (15) edge (8) edge (9);
        \node[vertex,fill=blue!30] (17) at (4,0) {} edge (16) edge (10) edge (9);
        \node[vertex] (18) at (6,0) {} edge (17) edge (11) edge (10);
        \node[vertex] (19) at (8,0) {} edge (18) edge (12) edge (11);
        \node[vertex] (20) at (10,0) {} edge (19) edge (12);
        \draw[orange,line width=2pt] (1,4) -- (9) -- (3,0);
        \node at (-5,2) {$\cdots$};
        \node at (10,2) {$\cdots$};
        \node at (11,2) {$=$};
        \begin{scope}[shift={(18.5,0)}]
            \node at (-6.5,2) {$\cdots$};
            \node at (8.5,2) {$\cdots$}; 
            \node[vertex] (a) at (-2,4) {};
            \node[vertex] (b) at (-4,4) {} edge (a);
            \node[vertex] (c) at (-3,2) {} edge (a) edge (b);
            \node[vertex] (d) at (-5,2) {} edge (b) edge (c);
            \node[vertex] (e) at (-4,0) {} edge (c) edge (d);
            \node[vertex] (f) at (-6,0) {} edge (d) edge (e);
            
            \node[vertex,fill=blue!30] (1) at (0,4) {} edge (a);
            \node[vertex,fill=blue!30] (2) at (2,4) {} edge (1);
            \node[vertex] (3) at (4,4) {} edge (2);
            \node[vertex] (4) at (6,4) {} edge (3);
            
            \node[vertex,fill=blue!30] (8) at (-1,2) {} edge (1) edge (c) edge (a);
            \node[vertex,draw=orange!120,line width=1pt,fill=blue] (9) at (1,2) {} edge (8) edge (1) edge (2);
            \node[vertex,fill=blue!30] (10) at (3,2) {} edge (9) edge (3) edge (2);
            \node[vertex] (11) at (5,2) {} edge (10) edge (4) edge (3);
            \node[vertex] (12) at (7,2) {} edge (11) edge (4);
            
            \node[vertex] (15) at (-2,0) {} edge (8) edge (e) edge (c);
            \node[vertex,fill=blue!30] (16) at (0,0) {} edge (15) edge (8) edge (9);
            \node[vertex,fill=blue!30] (17) at (2,0) {} edge (16) edge (10) edge (9);
            \node[vertex] (18) at (4,0) {} edge (17) edge (11) edge (10);
            \node[vertex] (19) at (6,0) {} edge (18) edge (12) edge (11);
            \node[vertex] (20) at (8,0) {} edge (19) edge (12);
            \draw[orange,line width=2pt] (1,4) -- (9) -- (1,0);
        \end{scope}
    \end{tikzpicture}
    \caption{First player move on $\oneslant{3}$ when $n$ is even, and its line of symmetry}
    \label{fig:trap_cen_odd}
    \end{figure}

    By making this first move and the given split, by \cref{thm:symmetry} Left wins.
\end{proof}
    
\vspace{5pt}

\begin{lemma}\label{lem:rightaddonly}
    When playing \snort on empty graphs of the form $\rightaddonly{3}$, the outcome class is a first player win.
\end{lemma}
\begin{proof}
    When $n=1$, the first player will play in the central vertex and prevent their opponent from moving. When $n=2$, the first player still plays in the central vertex $(2,2)$, after which the opponent can only respond in $(1,3)$, and the first player has several more moves and wins.
    
    When $n\geq3$, assume without loss of generality that Left is first to move.

    Consider the case when $n\geq 3$ is odd, as shown in \cref{fig:lat_long_even}.
    The optimal first move for Left in this situation is to move in $\left(\lceil n/2\rceil,2\right)$. Left then thinks of the graph as split by deleting the vertices $\left(\lceil n/2\rceil,2\right)$ and $\left(\lceil n/2\rceil,1\right)$ and the edge $\left(\left(\lceil n/2\rceil,3\right), \left(\lceil n/2\rceil+1,3\right)\right)$. 
    
    \begin{figure}[!ht]
    \centering
    \begin{tikzpicture}[scale=.4, vertex/.style={circle, draw, minimum size=2.5mm, font=\scriptsize}]
        
        \node[vertex] (a) at (0,4) {};
        \node[vertex] (b) at (-2,4) {} edge (a);
        \node[vertex] (c) at (0,2) {} edge (a) edge (b);
        \node[vertex] (d) at (-2,2) {} edge (b) edge (c);
        \node[vertex] (e) at (0,0) {} edge (c) edge (d);
        \node[vertex] (f) at (-2,0) {} edge (d) edge (e);
        
        \node[vertex,fill=blue!30] (2) at (2,4) {} edge (a);
        \node[vertex,draw=orange!120,line width=1pt,fill=blue!30] (3) at (4,4) {} edge (2);
        \node[vertex] (4) at (6,4) {} edge (3);
        \node[vertex] (5) at (8,4) {} edge (4);
        \node[vertex] (6) at (10,4) {} edge (5);
        
        \node[vertex,fill=blue!30] (9) at (2,2) {} edge (2) edge (a) edge (c);
        \node[vertex,draw=orange!120,line width=1pt,fill=blue] (10) at (4,2) {} edge (9) edge[orange,line width=2pt] (3) edge (2);
        \node[vertex,fill=blue!30] (11) at (6,2) {} edge (10) edge (4) edge (3);
        \node[vertex,fill=blue!30] (12) at (8,2) {} edge (11) edge (5) edge (4);
        \node[vertex] (13) at (10,2) {} edge (12) edge (6) edge (5);
        
        \node[vertex] (16) at (2,0) {} edge (9) edge (c) edge (e);
        \node[vertex,fill=blue!30] (17) at (4,0) {} edge (16) edge (10) edge (9);
        \node[vertex,fill=blue!30] (18) at (6,0) {} edge (17) edge (11) edge (10);
        \node[vertex] (19) at (8,0) {} edge (18) edge (12) edge (11);
        \node[vertex] (20) at (10,0) {} edge (19) edge (13) edge (12);
        \node[vertex] (21) at (12,0) {} edge (20) edge (13);
        \draw[orange,line width=2pt] (10) -- (5,0);
        \node at (-3,2) {$\cdots$};
        \node at (12,2) {$\cdots$};
        \node at (13,2) {$=$};
        \begin{scope}[shift={(18.5,0)}]
            \node at (-4.5,2) {$\cdots$};
            \node at (10,2) {$\cdots$}; 
            
            \node[vertex] (a) at (0,4) {};
            \node[vertex] (b) at (-2,4) {} edge (a);
            \node[vertex] (c) at (-1,2) {} edge (a) edge (b);
            \node[vertex] (d) at (-3,2) {} edge (b) edge (c);
            \node[vertex] (e) at (-2,0) {} edge (c) edge (d);
            \node[vertex] (f) at (-4,0) {} edge (d) edge (e);
            
            \node[vertex,fill=blue!30] (2) at (2,4) {} edge (a);
            \node[vertex,draw=orange!120,line width=1pt,fill=blue!30] (3) at (4,4) {} edge (2);
            \node[vertex] (4) at (6,4) {} edge (3);
            \node[vertex] (5) at (8,4) {} edge (4);
            \node[vertex] (6) at (10,4) {} edge (5);
            
            \node[vertex,fill=blue!30] (9) at (1,2) {} edge (2) edge (a) edge (c);
            \node[vertex,draw=orange!120,line width=1pt,fill=blue] (10) at (3,2) {} edge (9) edge[orange,line width=2pt] (3) edge (2);
            \node[vertex,fill=blue!30] (11) at (5,2) {} edge (10) edge (4) edge (3);
            \node[vertex,fill=blue!30] (12) at (7,2) {} edge (11) edge (5) edge (4);
            \node[vertex] (13) at (9,2) {} edge (12) edge (6) edge (5);
            
            \node[vertex] (16) at (0,0) {} edge (9) edge (c) edge (e);
            \node[vertex,fill=blue!30] (17) at (2,0) {} edge (16) edge (10) edge (9);
            \node[vertex,fill=blue!30] (18) at (4,0) {} edge (17) edge (11) edge (10);
            \node[vertex] (19) at (6,0) {} edge (18) edge (12) edge (11);
            \node[vertex] (20) at (8,0) {} edge (19) edge (13) edge (12);
            \node[vertex] (21) at (10,0) {} edge (20) edge (13);
            \draw[orange,line width=2pt] (10) -- (3,0);
            \end{scope}
    \end{tikzpicture}
    \caption{First player move on $\rightaddonly{3}$ when $n$ is odd, and its line of symmetry}
    \label{fig:lat_long_even}
    \end{figure}

    When $n\geq 4$ is even, as shown in \cref{fig:lat_long_odd}, the optimal first move for Left is to go in $\left(n/2+1,2\right)$. Left thinks of the graphs as split by deleting the vertices
    \begin{align*}
        \left(\frac{n}{2},1\right), \left(\frac{n}{2},2\right), \left(\frac{n}{2}+1,1\right), \left(\frac{n}{2}+1,2\right), \text{ and } \left(\frac{n}{2}+1,3\right).
    \end{align*}
    
    \begin{figure}[!ht]
    \centering
    \begin{tikzpicture}[scale=.4, vertex/.style={circle, draw, minimum size=2.5mm, font=\scriptsize}]
        
        \node[vertex] (1) at (0,4) {};
        \node[vertex] (2) at (2,4) {} edge (1);
        \node[vertex,draw=orange!120,line width=1pt,fill=blue!30] (3) at (4,4) {} edge (2);
        \node[vertex,draw=orange!120,line width=1pt,fill=blue!30] (4) at (6,4) {} edge[orange,line width=2pt] (3);
        \node[vertex] (5) at (8,4) {} edge (4);
        \node[vertex] (6) at (10,4) {} edge (5);
        
        \node[vertex] (8) at (0,2) {} edge (1);
        \node[vertex] (9) at (2,2) {} edge (8) edge (1) edge (2);
        \node[vertex,draw=orange!120,line width=1pt,fill=blue!30] (10) at (4,2) {} edge (9) edge[orange,line width=2pt] (3) edge (2);
        \node[vertex,draw=orange!120,line width=1pt,fill=blue] (11) at (6,2) {} edge[orange,line width=2pt] (10) edge[orange,line width=2pt] (4) edge[orange,line width=2pt] (3);
        \node[vertex,fill=blue!30] (12) at (8,2) {} edge (11) edge (5) edge (4);
        \node[vertex] (13) at (10,2) {} edge (12) edge (6) edge (5);
        
        \node[vertex] (15) at (0,0) {} edge (8);
        \node[vertex] (16) at (2,0) {} edge (15) edge (8) edge (9);
        \node[vertex] (17) at (4,0) {} edge (16) edge (10) edge (9);
        \node[vertex,draw=orange!120,line width=1pt,fill=blue!30] (18) at (6,0) {} edge (17) edge[orange,line width=2pt] (11) edge[orange,line width=2pt] (10);
        \node[vertex,fill=blue!30] (19) at (8,0) {} edge (18) edge (12) edge (11);
        \node[vertex] (20) at (10,0) {} edge (19) edge (13) edge (12);
        \node[vertex] (21) at (12,0) {} edge (20) edge (13);
        \node at (-1.5,2) {$\cdots$};
        \node at (12,2) {$\cdots$};
        \node at (13,2) {$=$};
        \begin{scope}[shift={(15.5,0)}]
            \node at (-1.5,2) {$\cdots$};
            \node at (11.5,2) {$\cdots$};
            
            \node[vertex] (1) at (1,4) {};
            \node[vertex] (2) at (3,4) {} edge (1);
            \node[vertex,draw=orange!120,line width=1pt,fill=blue!30] (3) at (5,4) {} edge (2);
            \node[vertex,draw=orange!120,line width=1pt,fill=blue!30] (4) at (7,4) {} edge[orange,line width=2pt] (3);
            \node[vertex] (5) at (9,4) {} edge (4);
            \node[vertex] (6) at (11,4) {} edge (5);
            
            \node[vertex] (8) at (0,2) {} edge (1);
            \node[vertex] (9) at (2,2) {} edge (8) edge (1) edge (2);
            \node[vertex,draw=orange!120,line width=1pt,fill=blue!30] (10) at (4,2) {} edge (9) edge[orange,line width=2pt] (3) edge (2);
            \node[vertex,draw=orange!120,line width=1pt,fill=blue] (11) at (6,2) {} edge[orange,line width=2pt] (10) edge[orange,line width=2pt] (4) edge[orange,line width=2pt] (3);
            \node[vertex,fill=blue!30] (12) at (8,2) {} edge (11) edge (5) edge (4);
            \node[vertex] (13) at (10,2) {} edge (12) edge (6) edge (5);
            
            \node[vertex] (15) at (-1,0) {} edge (8);
            \node[vertex] (16) at (1,0) {} edge (15) edge (8) edge (9);
            \node[vertex] (17) at (3,0) {} edge (16) edge (10) edge (9);
            \node[vertex,draw=orange!120,line width=1pt,fill=blue!30] (18) at (5,0) {} edge (17) edge[orange,line width=2pt] (11) edge[orange,line width=2pt] (10);
            \node[vertex,fill=blue!30] (19) at (7,0) {} edge (18) edge (12) edge (11);
            \node[vertex] (20) at (9,0) {} edge (19) edge (13) edge (12);
            \node[vertex] (21) at (11,0) {} edge (20) edge (13);
        \end{scope}
    \end{tikzpicture}
    \caption{First player move on $\rightaddonly{3}$ when $n$ is even, and its line of symmetry}
    \label{fig:lat_long_odd}
    \end{figure}

    Left then wins in both cases by \cref{thm:symmetry}.
\end{proof}

For the last variant, we will only be considering the case that $n$ is even. Although computational evidence indicates that the case when $n$ is odd is also a first player win, a similar copycat strategy cannot be used. The first player would have to make two moves at the start of the game to be able to split the board into two isomorphic, aside from shading, components. Instead, the optimal moves for the first player were to play greedy, i.e.\ reserve as many vertices in each move as possible. Even for small $n$, proving that this guarantees they can win has to consider many different cases and is tedious.

\begin{lemma}\label{lem:rightminusonly}
    When playing \snort on empty graphs of the form $\rightminusonly{3}$ where $n$ is even, the outcome class is a first player win.
\end{lemma}
\begin{proof}
    Assume without loss of generality that Left goes first and $n\geq 2$, see \cref{fig:lat_short_even}. The first move for Left in this situation is to go in $\left(n/2+1,2\right)$, which allows for Left to think of the graph as if it were split into two isomorphic graphs, aside from some shading, by the one hand tied principle. This can be done by deleting the vertices
    \begin{align*}
        \left(\frac{n}{2},1\right), \left(\frac{n}{2}+1,1\right), \left(\frac{n}{2}+1,2\right), \text{ and } \left(\frac{n}{2}+1,3\right).
    \end{align*}
    
    \begin{figure}[!ht]
    \centering
    \begin{tikzpicture}[scale=.4, vertex/.style={circle, draw, minimum size=2.5mm, font=\scriptsize}]
        
        \node[vertex] (1) at (0,4) {};
        \node[vertex] (2) at (2,4) {} edge (1);
        \node[vertex,draw=orange!120,line width=1pt,fill=blue!30] (3) at (4,4) {} edge (2);
        \node[vertex,draw=orange!120,line width=1pt,fill=blue!30] (4) at (6,4) {} edge[orange,line width=2pt] (3);
        \node[vertex] (5) at (8,4) {} edge (4);
        \node[vertex] (6) at (10,4) {} edge (5);
        
        \node[vertex] (8) at (0,2) {} edge (1);
        \node[vertex] (9) at (2,2) {} edge (8) edge (1) edge (2);
        \node[vertex,fill=blue!30] (10) at (4,2) {} edge (9) edge (3) edge (2);
        \node[vertex,draw=orange!120,line width=1pt,fill=blue] (11) at (6,2) {} edge (10) edge[orange,line width=2pt] (4) edge[orange,line width=2pt] (3);
        \node[vertex,fill=blue!30] (12) at (8,2) {} edge (11) edge (5) edge (4);
        \node[vertex] (13) at (10,2) {} edge (12) edge (6) edge (5);
        \node[vertex] (14) at (12,2) {} edge (13) edge (6);
        
        \node[vertex] (15) at (0,0) {} edge (8);
        \node[vertex] (16) at (2,0) {} edge (15) edge (8) edge (9);
        \node[vertex] (17) at (4,0) {} edge (16) edge (10) edge (9);
        \node[vertex,draw=orange!120,line width=1pt,fill=blue!30] (18) at (6,0) {} edge (17) edge[orange,line width=2pt] (11) edge (10);
        \node[vertex,fill=blue!30] (19) at (8,0) {} edge (18) edge (12) edge (11);
        \node[vertex] (20) at (10,0) {} edge (19) edge (13) edge (12);
        \node[vertex] (21) at (12,0) {} edge (20) edge (14) edge (13);
        \node at (-1.5,2) {$\cdots$};
        \node at (13,2) {$\cdots$};
        \node at (14,2) {$=$};
        \begin{scope}[shift={(16.5,0)}]
            \node at (-1.5,2) {$\cdots$};
            \node at (13,2) {$\cdots$};
            
            \node[vertex] (1) at (1,4) {};
            \node[vertex] (2) at (3,4) {} edge (1);
            \node[vertex,draw=orange!120,line width=1pt,fill=blue!30] (3) at (5,4) {} edge (2);
            \node[vertex,draw=orange!120,line width=1pt,fill=blue!30] (4) at (7,4) {} edge[orange,line width=2pt] (3);
            \node[vertex] (5) at (9,4) {} edge (4);
            \node[vertex] (6) at (11,4) {} edge (5);
            
            \node[vertex] (8) at (0,2) {} edge (1);
            \node[vertex] (9) at (2,2) {} edge (8) edge (1) edge (2);
            \node[vertex,fill=blue!30] (10) at (4,2) {} edge (9) edge (3) edge (2);
            \node[vertex,draw=orange!120,line width=1pt,fill=blue] (11) at (6,2) {} edge (10) edge[orange,line width=2pt] (4) edge[orange,line width=2pt] (3);
            \node[vertex,fill=blue!30] (12) at (8,2) {} edge (11) edge (5) edge (4);
            \node[vertex] (13) at (10,2) {} edge (12) edge (6) edge (5);
            \node[vertex] (14) at (12,2) {} edge (13) edge (6);
            
            \node[vertex] (15) at (-1,0) {} edge (8);
            \node[vertex] (16) at (1,0) {} edge (15) edge (8) edge (9);
            \node[vertex] (17) at (3,0) {} edge (16) edge (10) edge (9);
            \node[vertex,draw=orange!120,line width=1pt,fill=blue!30] (18) at (5,0) {} edge (17) edge[orange,line width=2pt] (11) edge (10);
            \node[vertex,fill=blue!30] (19) at (7,0) {} edge (18) edge (12) edge (11);
            \node[vertex] (20) at (9,0) {} edge (19) edge (13) edge (12);
            \node[vertex] (21) at (11,0) {} edge (20) edge (14) edge (13);
        \end{scope}
    \end{tikzpicture}
    \caption{First player move on $\rightminusonly{3}$ when $n$ is even, and its line of symmetry}
    \label{fig:lat_short_even}
    \end{figure}

    Using \cref{thm:symmetry}, Left will play last and win. Thus the game is a first player win.
\end{proof}

Summarizing these results we get the following:
\begin{theorem}\label{thm:2rows}
    When playing \snort on empty graphs of the form $\basegraph{3}$ or its variations $\leftaddoneboth{3}$, $\bothaddone{3}$, $\bothminusone{3}$, $\oneslant{3}$, $\rightaddonly{3}$, and, when $n$ is even, $\rightminusonly{3}$, the outcome class is a first player win.
\end{theorem}

\section{Future Work}

Finding an alternate strategy for the first player in $\rightminusonly{3}$ when $n$ is odd would be the obvious next step, but this appears to be difficult.

Since we only considered variations in which a full triangle was added to the end of the triangular grid $T_{n,2}$ and $T_{n,3}$, there are many more subgraphs of the infinite triangular grid that can be considered. For example, adding a leaf to either graph, or adding a chain of triangles to $T_{n,3}$, as well as the triangular grid $T_{n,m}$ for larger values of $m$. A particular case of the latter is the diamond $T_{n,n}$, and preliminary work indicates these graphs are also first player wins when playing \snort by using more than one axis of symmetry.

Additionally, we only considered when the graphs are connected -- if there are disconnected components, the outcome class of each connected component alone is not enough to determine the outcome class of the overall game. In combinatorial game theory, a refinement of the outcome class known as the game value is used in this case. We also completed preliminary work computing game values for \snort played on the graphs considered in this paper, but they appear to become too complicated too quickly to be able to describe them succinctly.

\end{document}